\documentclass[francais]{cedram-aif}
\usepackage[english,frenchb]{babel}
\usepackage[latin1]{inputenc}
 \usepackage[T1]{fontenc}
 \usepackage{lmodern}
\usepackage[active]{srcltx}
\usepackage{graphicx}
\usepackage{graphicx,epsfig,amscd,graphics,xypic}
% ...

\newcommand{\TT}{T{\footnotesize H\'EOR\`EME} }
\newcommand{\PP}{P{\footnotesize ROPOSITION} }
\newcommand{\CC}{C{\footnotesize OROLLAIRE} }
\newcommand{\DD}{D{\footnotesize \'EFINITION} }
\newcommand{\LL}{L{\footnotesize EMME} }
\newcommand{\EE}{E{xemple} }

\title
[Sur la linéarisation des tissus]
{Sur la linéarisation des tissus}

\alttitle{On the linearizability of webs}

\author{\firstname{Luc}  \lastname{Pirio}}

\address{IRMAR, UMR 6625 du CNRS\\
Université Rennes 1\\
 Campus de Beaulieu\\
35000 Rennes.}

\email{E-mail: luc.pirio@univ-rennes1.fr}

  
\altkeywords{web, linearization,  projective connection, curvature}


\begin{document}
\begin{abstract}
Nous donnons un critère analytique simple qui caractérise les tissus en hypersurfaces  linéarisables. On en donne  une interprétation géométrique invariante en terme de connexion projective. Nous expliquons que notre approche permet de traiter le problème de la linéarisation  d'objets plus généraux que les tissus de codimension 1. Quelques exemples particuliers sont étudiés ensuite en guise d'illustration. Nous concluons finalement par des remarques historiques sur la question abordée ici.
\end{abstract}

\begin{altabstract}
 We give a simple analytic criterion which characterizes linearizable 1-codimensional webs. Then we give an invariant geometrical interpretation of it, in term of projective connection. We explain then how our approach allows to study linearization of more general objects than 1-codimensional webs. By way of illustration, we treat some explicit interesting examples. 
We finish by historic remarks on the question approached here.
\vspace{-0.5cm}

\end{altabstract}

\maketitle


\section{Introduction}
On peut faire remonter 
la question de la linéarisation des tissus  au problème de l'anamorphose en nomographie. 
 Cette problématique  date donc de  plus d'un siècle. La géométrie des tissus s'est constituée comme discipline autonome à Hambourg vers la fin des années 1920: Blaschke et ses collaborateurs 
 ont alors établi un nombre important de résultats, qui ont constitué un cadre solide sur lequel s'est b\^atie la théorie par la suite. Le problème de savoir si un tissu donné est linéarisable ou pas a  bien sûr retenu l'attention de ces géomètres qui en ont donné une solution satisfaisante en dimension 2. 
 Cette question a été reprise récemment dans ce même cadre par différents auteurs, de façon indépendante semble-t-il.  
\`A notre connaissance, il n'y a pas de résultat général concernant la linéarisation des tissus qui ne sont pas plans.\footnote{Nous renvoyons le lecteur à la Section  \ref{S:histoire} pour une discussion sur ce point.}
\smallskip

Dans cet article, nous donnons une solution complète au problème de caractériser les $d$-tissus de codimension 1 linéarisables en  dimension $n$ arbitraire, lorsque $d$ est plus grand que $ n+2$. Il nous semble que   l'approche développée ici  peut se généraliser sans véritable difficulté à la caractérisation des tissus en toute dimension et/ou codimension. La difficulté qu'il y a à savoir si un tissu est linéarisable ou pas provient du fait qu'un tel objet  est la donnée d'un ensemble discret de feuilletages, alors que les méthodes  classiques permettant d'étudier l'équivalence analytique (locale) d'objets géométriques sont du ressort de la géométrie différentielle et  traitent donc  des structures géométriques continues. Il se trouve que cette difficulté disparaît si l'on ``interpole'' \underline{tous} les feuilletages d'un tissu donné par \underline{un} même système différentiel d'ordre 2, puisqu'il est immédiat qu' 
\begin{quote}
{\it un tissu est linéarisable si et seulement si les feuilles des différents feuilletages qui le constituent sont des variétés intégrales d'un même système différentiel du second ordre (avec $1\leq \alpha \leq \beta <n$)
\begin{equation}
 \label{E:1}
 \frac{\partial^2 x^n}{\partial x^\alpha \partial x^\beta}= 
F_{\alpha \beta}(x) 
\end{equation}
linéarisable, c'est-à-dire équivalent au système {\og{plat}\fg}
\begin{equation}
 \label{E:2}
 \frac{\partial^2 \overline{x}^n}{\partial \overline{x}^\alpha \partial \overline{x}^\beta}= 0  \quad  \quad 
\end{equation}
via un changement de coordonnées ponctuelles $x=(x^1,\ldots,x^n)\longmapsto (\overline{x}^1(x),\ldots,\overline{x}^n(x))$. 
}
\end{quote}

C'est sur ce principe particulièrement simple que repose notre approche. En effet, la caractérisation des systèmes du second ordre  (\ref{E:1}) équivalents par changement de coordonnées ponctuelles au système plat (\ref{E:2}) est classique et bien comprise. Elle peut se formaliser en terme de connexion projective. L'utilisation alors de la notion d'espace généralisé introduite par \'E. Cartan permet d'obtenir facilement un critère caractérisant les connexions projectives intégrables, c'est-à-dire  (localement) équivalentes à la connexion projective plate.  

La question de linéariser un $d$-tissu $W_d$ se ramène alors à celle de trouver une connexion projective ({\it i.e.} un système de la forme (\ref{E:1}) sous une certaine forme normale) interpolant tous les feuilletages de $W_d$. Ce problème est  en fait très simple puisqu'équivalent à la résolution d'un système d'équations linéaires, système  qui se trouve être surdéterminé lorsque $d$ est suffisamment grand. 
\begin{center}
 $\Diamond$\vspace{0.1cm}
\end{center}

C'est en essayant d'interpréter avec un formalisme plus moderne le chapitre 27 du livre de référence \cite{bb} que nous avons compris que la bonne notion pour étudier la question de la linéarisabilité des tissus était la notion de connexion projective. Nous avons ensuite réalisé que l'approche de \cite{bb} se généralisait sans véritable difficulté au cas des tissus de codimension 1. Nous avons écrit cet article sans prétendre à une grande originalité mais en espérant faire un exposé clair sur une question naturelle en géométrie des tissus. 

\begin{center}
 $\Diamond$\vspace{0.1cm}
\end{center}

Indiquons maintenant  comment est organisée la suite de cet article. \smallskip

Dans la Section 2, on commence par énoncer le \TT \ref{T:main} que nous présentons comme notre résultat principal: il s'agit d'un critère caractérisant de façon invariante les $(n+2)$-tissus de codimension 1   linéarisables. On énonce aussi le \CC \ref{C:main} qui est une adaptation immédiate du résultat précédent au cas des $d$-tissus lorsque $d>n+2$. On introduit ensuite en \S2.2 les notions et  notations qui seront utilisées dans la suite. \medskip

L'outil essentiel dans cet article est la notion de connexion projective: on la présente sous différents aspects dans la Section \ref{S:connexionprojo} et l'on donne les ingrédients utilisés pour  caractériser les connexions projectives intégrables par l'annulation de leur courbure. Cette section ne comporte pas de résultat qui ne soit déjà   connu et le spécialiste n'y apprendra  rien de nouveau. Par contre, elle sera 
peut-être utile au   lecteur peu familier avec les connexions projectives, qui y trouvera une présentation que l'on espère assez complète et bien référencée. Dans tous les cas, cette section peut être sautée en première lecture. \medskip

Dans la Section \ref{S:codim1}, nous montrons que les feuilles d'un $(n+2)$-tissu $W$  de codimension 1 sont totalement géodésiques pour une certaine connexion projective $\Pi_W$ uniquement déterminée qui se trouve  donc être canoniquement associée à $W$. Combiné avec les résultats de la section précédente, on en déduit assez simplement une preuve de notre résultat principal.  \medskip

Le point de vue adopté dans cet article permet  d'obtenir dans certains cas des critères de linéarisabilité pour des objets  plus généraux que les tissus de codimension 1. C'est ce thème qui est discuté dans la Section \ref{S:gen}. 
En guise d'illustration, nous traitons plusieurs exemples différents. 
\medskip

Enfin, nous présentons dans la Section \ref{S:histoire} quelques remarques historiques sur  différents travaux antérieurs au présent article qui avaient abordé  la question de la linéarisabilité des tissus.


\section{\'Enoncé du résultat principal et notations}
Dans tout l'article, on se place dans un cadre holomorphe. Tous les objets considérés seront donc analytiques complexes. 
Nos résultats sont cependant valides dans un cadre réel, sous des hypothèses de régularité $C^k$ pour $k\geq 2$. Nous laissons le soin au lecteur intéressé d'adapter 
 et/ou de modifier les énoncés/preuves de cet article à d'autres situations.
\subsection{\'Enoncés du résultat principal et d'un corollaire}\vspace{-0.15cm}
Nous dirons qu'une connexion projective $\Pi$ et un tissu $W$ sont {\it compatibles} si les feuilles de $W$ sont toutes totalement géodésiques pour $\Pi$.
 Cette définition posée, notre résultat principal se formule de la façon~suivante:
\begin{thm}
\label{T:main}
 \'Etant donné un $(n+2)$-tissu  $W$ de codimension 1 sur un domaine $U\subset \mathbb C^n$, il existe une unique connexion projective $\Pi_W$ qui lui est compatible. 
De plus, les assertions suivantes sont équivalentes:
\begin{itemize}
 \item $W$ est linéarisable;
\item $\Pi_W$ est plate;
\item la courbure $\mathfrak{ C}(W) \in \Omega^1_U \otimes \mathfrak{ sl}_{n+1}(\mathbb C)$ de $\Pi_W$ est identiquement nulle.
\end{itemize}
\end{thm}
Ce théorème se démontre très simplement: on commence par établir  l'existence et l'unicité de la connexion projective $\Pi_W$ compatible avec un $(n+2)$-tissu  $W$ donné. C'est la \PP \ref{P:existePI} dont la preuve repose essentiellement sur de l'algèbre linéaire. Le critère de linéarisation énoncé dans le \TT \ref{T:main} provient alors du critère (classique) caractérisant les connexions projectives plates  que l'on aura rappelé auparavant dans la Section \ref{S:connexionprojo} ({\it cf.} le \TT  \ref{T:caractrisationCARTANplat} ainsi que le \CC \ref{C:caractCNXIONprojoplate1}). 


 Du théorème précédent, on  déduit immédiatement le 
\begin{coro} 
\label{C:main}
Soit $d\geq n+2$. 
 Un   $d$-tissu  en hypersurfaces sur  un domaine $U$ de $\mathbb C^n$
 est linéarisable si et seulement si toutes les connexions projectives de ses $(n+2)$-sous-tissus coïncident et sont plates. 
\end{coro}

  Dans la Section \ref{S:codim1}, nous donnons  différentes formulations invariantes  des conditions de linéarisation qui apparaissent dans le \TT \ref{T:main} et le \CC \ref{C:main}. 
\vspace{-0.5cm}

\subsection{Notations}
\label{S:notations}\vspace{-0.15cm}
La liste ci-dessous présente certaines des notations   utilisées dans l'article:
\smallskip
\begin{tabbing}
aaaaaaaaaaaa \= Description \kill
$n$\> entier plus grand que 2;\\
$i,j,k,l$ \> indices variant entre $1$ et $n$;\\
$ \alpha, \beta, \gamma,\delta$ \> indices variant entre $1$ et $n$;\\
$M$ \> variété complexe connexe de dimension $n$;\\
$U$ \>  domaine ({\it i.e} ouvert connexe et simplement connexe) de $ M$ identifié à  un domaine  de $ \mathbb C^n$;   \\
$T_M$ \> fibré tangent de $M$;\\
$\Omega^k_M$ \> $k$-ième puissance extérieure du fibré cotangent de $M$;\\
$x^1,\ldots,x^n$ \> système de coordonnées holomorphes sur $U$;\\
$\nabla$ \> connexion affine sur $M$;\\
$\Gamma_{ij}^k$ \> symboles de Christofell de $\nabla$ dans les coordonnées
$x^1,\ldots,x^n$;\\
$[\nabla]$\> classe d'équivalence projective de $\nabla$;\\
$\Pi$ \> connexion projective sur $M$;\\
$\Pi_{ij}^k$ \> coefficients de Thomas de  $\Pi$ dans les coordonnées
$x^1,\ldots,x^n$;\\
$D( \mathbb C^n,0)$ \>  groupe $Dif\!f( \mathbb C^n,0)$  des germes de biholomorphismes holomorphes en $({\mathbb C}^n,0)$;\\
$G$ \> groupe de Lie complexe;\\
$H$ \> sous-groupe de Lie fermé de  $G$;\\
$\mathfrak g,\mathfrak h$ \> algèbres de Lie de $G$ et $H$ (respectivement);\\
$\rho$ \> application de passage au quotient $\rho:\mathfrak g\rightarrow \mathfrak g/\mathfrak h$;\\
$P$ \> $H$-fibré principal à droite sur $M$;\\
$\omega$ \> $\mathfrak g$-connexion de Cartan sur $P$;\\
$\Omega={\mathfrak C}(\omega)$ \> courbure de  $\omega$;\\
$\mathfrak t(\omega)$ \> torsion $\rho(\Omega)$ de $\omega$;\\
$\sigma$ \> jauge pour le fibré $P|_U\rightarrow U$;\\
$\omega_\sigma$ \>  forme de soudure $\sigma^* \omega$ associée à $\sigma$;\\
$\mathcal F$ \> feuilletage ;\\
$W$ \> tissu;\\
$d$\> nombre de feuilletages formant $W$;\\
$W_d$ \> $d$-tissu.

\end{tabbing}
\smallskip

Quelques commentaires et précisions s'imposent: l'ouvert $U$ désignera le plus souvent un voisinage ouvert connexe et simplement connexe de l'origine dans ${\mathbb C}^n$ et on aura souvent $M=U$. L'expression {\it   système de coordonnées  $x^1,\ldots,x^n$ sur $U$} fera référence aux composantes $x^i$ d'un biholomorphisme ${\bf x}:U\stackrel{\sim} {\rightarrow} {\bf x}(U)\subset \mathbb C^n$.  Comme l'on s'intéresse au problème de  caractériser les tissus  linéarisables par un biholomorphisme local, l'on se permettra de restreindre $U$ autour de l'origine aussi souvent qu'il sera nécessaire de le faire. \smallskip

Si $F$ est un fibré vectoriel sur $M$, on commettra l'abus d'écrire $F$ au lieu de $H^0(M,F)$: par exemple, la formule (\ref{E:weyltenseur}) doit se comprendre comme $ \mathfrak W_\Pi \in H^0(U,T_ U\otimes (\Omega^1_U)^{\otimes 3})$. 
\smallskip

Des sous-espaces vectoriels  $E_1,\ldots,E_d$ d'un  $\mathbb C$-espace vectoriel complexe $E$ de dimension $n$ sont dits {\it en position générale} si pour toute partie $A\subset \{1,\ldots,d\}$ on a 
${\rm dim} (  \sum_{a\in A} E_i ) = \min (  n,   \sum_{a\in A} {\rm dim} (E_a ) )$, ainsi que  ${\rm codim} (  \cap_{a\in A} E_A  ) = \min ( n, \sum_{a\in A}{\rm codim}(E_a))  $. 
 \`A chaque sous-espace $E_a$  de codimension $c_a$ est associée une {\it normale} $\mathfrak n_a\in \wedge ^{c_a} (E)$ qui est unique (modulo multiplication par une constante non-nulle): il suffit de prendre $\mathfrak n_a=\ell^1\wedge \cdots \wedge \ell^{c_a}$, où $\ell^1,\ldots,\ell^{c_a}$ désigne une base arbitraire de ${\rm Ann}(E_a)\subset E^*$. La donnée de $\mathfrak n_a$ (modulo multiplication)  est équivalente à celle de $E_a$ puisque  $E_a=\{ e \in E \, | \, i_e( \mathfrak n_a )\equiv 0 \, \}$.  Si $n=\sum_{a=1}^d {\rm dim}(E_a)$, l'hypothèse de position générale correspond à une décomposition en somme directe $E=E_1\oplus \cdots \oplus E_d$, ce qui équivaut à $ 
 \mathfrak n_1 \wedge \cdots \wedge \mathfrak n_d \neq 0$. \smallskip

Soient $\mathcal F_1,\ldots,\mathcal F_d$ des feuilletages holomorphes réguliers sur $M$, de codimensions respectives $c_1,\ldots,c_d$. Chacun d'eux  est défini par une {\it normale} $\Omega_a\in \Omega^{c_a}_M$ qui est unique modulo multiplication par une section partout non-nulle d'un fibré en droites $L_a$ sur $M$. Par définition, le $d$-uplet $W_d=( \mathcal F_1,\ldots,\mathcal F_d)$ est un {\it $d$-tissu} (ordonné) sur $M$ si les espaces tangents  $T_{\mathcal F_a,m}$ (pour $a=1,\ldots,d$) sont en position générale dans $T_{M,m}$, cela quel que soit le point $m\in M$. Un tissu $W$ est dit {\it mixte} si les codimensions de ses feuilletages ne sont pas toutes égales. Si $c=c_1=\cdots=c_d$, on parlera simplement de {\it $d$-tissu de codimension $c$}. Lorsque $c=1$ et $c=n-1$, on parlera aussi de tissus en hypersurfaces et de tissus en courbes (respectivement). 
 Enfin, lorsque $c$ divise la dimension $n$ de l'espace ambiant ({\it i.e} lorsque $n=k\,c$ pour un entier $k$),  l'hypothèse de position générale se traduit par 
$$
\Omega_{a_1}\wedge \cdots \wedge \Omega_{a_k}\neq 0
$$
quels que soient $a_1,\ldots,a_k$ tels que $1\leq a_1<a_2<\ldots<a_k\leq d$.\smallskip

Un tissu sur $U\subset \mathbb C^n$ est {\it linéaire} si les feuilles de ses feuilletages sont des (morceaux de) sous-espaces affines de $ \mathbb C^n$. Un tissu $W$ défini sur $M$ est {\it linéarisable} s'il existe un germe de biholomorphisme $\varphi:(M,m)\rightarrow (\mathbb C^n,\varphi(m))$ tel que  $\varphi_* (W)$ soit un germe de tissu linéaire. C'est une notion qui a tendance à se globaliser\footnote{Cela signifie ceci: si le germe d'un tissu  $W$ en un point $m\in M$ est linéarisable, alors il  l'est également  en tout point de $M$.} mais que l'on ne regardera ici que d'un point de vue local. 


\section{Connexions projectives}
\label{S:connexionprojo}
La notion de connexion projective est assez ancienne. Plusieurs auteurs voient sa naissance en 1921 dans le papier de Weyl \cite{weyl} qui a induit quasiment immédiatement un grand nombre d'articles sur le sujet. 
Deux courants d'études ont alors émergé. L'un se trouve souvent désigné comme  l'{\og 
\'Ecole de Princeton \fg} (constituée de 
Berwald, Eisenhart, Thomas, Veblen, Weyl, Whitehead parmi les plus connus), l'autre est principalement le fait d'\'Elie Cartan. 
Le point de vue de Cartan sur les connexions projectives est plus abstrait et conceptuel que celui de l'\'Ecole de Princeton. C'est un cas particulier des  notions d'{\og espace généralisé\fg}  et de {\og  connexion de Cartan\fg} introduites par  Cartan. L'approche de l'\'Ecole de Princeton à cet avantage d'être plus élémentaire et plus immédiate  pour calculer en coordonnées locales. Le lecteur intéressé pourra consulter \cite{borel} pour une comparaison des approches de l'école de Princeton et de celle de Cartan
. \smallskip

\subsection{Définition {\og à la Princeton\fg} via les connexions affines}
\label{S:viaconnexionaffine}
\subsubsection{Rappels sur les connexions affines}
Dans tout ce qui suit, l'on désigne par  $\nabla$ une connexion affine (holomorphe) sur $U$, c'est à dire un morphisme de fibrés vectoriels  $\nabla: T_U \rightarrow  \Omega_U^1\otimes T_U $ vérifiant l'identité~de~Leibniz 
$ \nabla(f X)=X\otimes df+f \nabla(X) $
pour tout germe de fonction holomorphe $f$ et tout germe de champ de vecteurs $X$ sur $U$. 
Un système de coordonnées $x^1,\ldots,x^n$ étant fixé sur $U$,  $\nabla$ est complètement définie par la donnée de ses symboles de Christoffel $\Gamma_{jk}^i$ définis par les relations 
$ \nabla\big( \frac{\partial }{\partial x^j}\big) = \Gamma_{ij}^k  \, dx^i \otimes \frac{\partial }{\partial x^k }$ pour $j=1,\ldots,n$. 
Les symboles de Christoffel $\overline{\Gamma}_{\beta\gamma}^\alpha$  de $\nabla$ relativement à un autre système de coordonnées $\overline{x}^1,\ldots, \overline{x}^n$ sont reliés aux $\Gamma_{jk}^i$ par les formules classiques 
(pour $\alpha,\beta,\gamma=1,\ldots,n$):
\begin{equation}
\label{E:transfochristofell}
 \overline{\Gamma}_{\alpha \beta}^\gamma= \Gamma_{ij}^k \frac{\partial x^i }{\partial \overline{x}^\alpha }\frac{\partial  x^j}{\partial \overline{x}^\beta}\frac{\partial 
\overline{x}^\gamma
}{\partial x^k}+ 
\frac{\partial^2 x^i }{\partial \overline{x}^\alpha \partial \overline{x}^\beta } \frac{\partial \overline{x}^\gamma}{\partial x^i}\,.
\end{equation}

La {\it torsion} de $\nabla$ est le tenseur $T$ deux fois contravariant tel que $\nabla_Y(X)-\nabla_X(Y)+[X,Y]+
T(X,Y)=0$ 
 quels que soient les champs de vecteurs $X$ et $Y$.  On vérifie que $\nabla$ est à torsion nulle si et seulement si elle est {\it symétrique}, c'est-à-dire si dans un (et donc dans tous) système(s) de coordonnées ses coefficients de Christoffel sont symétriques par rapport aux deux indices inférieurs. 

\begin{defi}
 Une {\it paramétrisation géodésique} (ou plus simplement, une ``géodésique'') pour la connexion $\nabla$ est une application holomorphe $c:t\mapsto c(t)$  définie sur un ouvert de $\mathbb C$ à valeurs dans $U$,~telle~que 
\begin{equation}\label{E:geod}
\nabla_{\frac{dc}{dt} }(  \frac{dc}{dt} ) \equiv 0.
\end{equation} 
\end{defi}

Si $c$  paramétrise une courbe et s'écrit $c(t)=(c^1(t),\ldots,c^n(t))$ dans des coordonnées $x^1,\ldots,x^n$, la relation 
(\ref{E:geod})  équivaut  
au système d'équations différentielles ordinaires du second ordre
\begin{equation}
\label{E:equatgeod}
\qquad 
 \frac{d^2c^k}{dt^2}+ \Gamma_{ij}^k\frac{dc^i}{dt} \frac{dc^j}{dt} \equiv 0 \qquad k=1,\ldots,n.
\end{equation}
On déduit facilement de la forme du système d'équations (\ref{E:equatgeod}) que, étant donné un vecteur tangent $\tau\in T_{U,u} $ en un point $u\in U$, il existe un unique germe de paramétrisation géodésique $c:(\mathbb C,0)\rightarrow (U,u)$ tel que $\frac{dc}{dt}|_{t=0}=\tau$.

\begin{defi}
 Une {\it courbe géodésique}  pour 
$\nabla$ est une courbe 
qui admet (localement) une paramétrisation géodésique. 
\end{defi}

\begin{prop} 
\label{P:systemdiffCgeodesique}
Soit $C\subset U$ une courbe. Les conditions suivantes sont équivalentes: 
\begin{enumerate}
 \item $C$ est une courbe géodésique pour la connexion $\nabla$;
\item il existe une paramétrisation $t\mapsto (c^1(t),\ldots,c^n(t))  $ de $C$ qui satisfait les équations différentielles
 \begin{equation}
\label{E:eqnonparamGEOD}
 \frac{dc^\ell}{dt} \Big( {\frac{d^2c^k}{dt^2}+ \Gamma_{ij}^k\frac{dc^i}{dt} \frac{dc^j}{dt}}{ 
}\Big)    =
\frac{dc^k}{dt}
 \Big( {\frac{d^2c^\ell}{dt^2}+ \Gamma_{ij}^\ell\frac{dc^i}{dt} \frac{dc^j}{dt}}
\Big);
\end{equation}
quels que soient $k,\ell=1,\ldots,n$;
\item  toute paramétrisation $t\mapsto (c^1(t),\ldots,c^n(t))  $ de $C$ satisfait les équations différentielles~(\ref{E:eqnonparamGEOD}).
\end{enumerate}
\end{prop}
De ci-dessus, on déduit que, étant donnée une direction tangente $L\in \mathbb PT_{U,u}$ en  $u\in U$, il existe une unique  courbe géodésique $C$ passant par $u$  telle que $T_{C,u}=L$. Nous dirons que $C$ est la {\it courbe géodésique  issue de $L$}.


\subsubsection{Connexions affines projectivement équivalentes}

\begin{defi}
\label{D:defequivprojo}
 Deux connexions affines $\nabla$ et $\overline{\nabla}$ sur $U$ sont projectivement équivalentes si elles ont les mêmes courbes géodésiques. 
\end{defi}

\begin{prop}[\cite{eisenhart,weyl}]
\label{P:cnxprojoWEYL}
Deux connexions affines $\nabla$ et $\overline{\nabla}$ sur $U$ sont projectivement équivalentes si et seulement si il existe des quantités $\varphi_1,\ldots,\varphi_n$ telles que les symboles de Christoffel $\Gamma_{ij}^k$ et $\overline{\Gamma}_{ij}^k$ de $\nabla$ et $\overline{\nabla}$  (par rapport au même système de coordonnées) vérifient 
\begin{equation}
 \overline{\Gamma}_{ij}^k=
\Gamma_{ij}^k+\delta_{i}^k \varphi_j+\delta_{j}^k \varphi_i
\end{equation}
quels que soient $i,j,k=1,\ldots,n$. 
\end{prop}
On vérifie immédiatement que l'équivalence projective définit bien une relation d'équivalence sur l'espace des connexions affines. On peut donc poser la 
\begin{defi}
\label{D:defconnexionprojective}
 Une connexion projective $\Pi$ sur $U$ est la donnée d'une classe d'équivalence projective de connexions affines sur $U$. On note $\Pi=[\nabla]$ si $\nabla$ est un représentant de $\Pi$.
\end{defi}
On vient de définir la   notion de connexion projective sur une variété complexe isomorphe à un domaine de $\mathbb C^n$. La notion générale de connexion projective s'en déduit  par recollement. 
\begin{defi}
\label{D:defconnexionprojective2}
 Une connexion projective $\Pi$ sur $M$ est la donnée d'une classe d'équivalence\footnote{La notion d'équivalence dont il est question ici est claire vu le contexte. Nous laissons le soins au lecteur de l'expliciter.}  d'atlas  $ (U_i, [\nabla_i])_{i\in I} $ sur $M $ tels que
\begin{enumerate}
  \item  $(U_i)_{i\in I}$ est un recouvrement  de $M$ par des ouverts isomorphes à des domaines de $\mathbb C^n$;
\item pour tout $i\in I$, $[\nabla_i]$ est une connexion projective sur $U_i$ (au sens de la \DD \ref{D:defconnexionprojective});
\item on a  $[\nabla_i]=[\nabla_j]$ sur  toute intersection $U_i\cap U_j$ non-vide.
 \end{enumerate}
\end{defi}
\begin{rema}
La définition ci-dessus s'étend bien sûr au réel mais est alors maladroite puisque la \DD \ref{D:defconnexionprojective} suffit dans ce cadre. En effet,  si $\Pi_{\mathbb R}$ désigne une connexion projective réelle sur une variété réelle $M_{\mathbb R}$ arbitraire (seulement supposée paracompacte pour assurer l'existence de partitions de l'unité),  il existe toujours une connexion affine réelle  $\nabla_{\mathbb R}$ sur $M_{\mathbb R}$ telle qu'on puisse écrire globalement $\Pi_{\mathbb R}=[\nabla_{\mathbb R}]$. Dans le cadre holomorphe, cela est vrai si et seulement si le 1-cocycle (3.2) de \cite{kobayashiochiai} est nul dans le groupe de cohomologie $H^1(M,\Omega^1_M)$, ce qui est vérifié si et seulement si 
le fibré canonique $K_M=\Omega^n_M$ de $M$ peut être muni d'une 
 connexion affine holomorphe (comme il  ressort de la lecture des   pages 96-97 de \cite{gunning}).\footnote{L'auteur remercie Sorin Dumitrescu pour avoir attiré son attention sur ce point.}
\end{rema}

\begin{prop}
 Une connexion affine $\nabla$ étant donnée, il existe une unique connexion symétrique  admettant les mêmes paramétrisations géodésiques. C'est la connexion $\overline{\nabla}$ définie par les symboles de Christofell 
$$\overline{\Gamma}_{ij}^k=\frac{1}{2} \big({\Gamma}_{ij}^k+  {\Gamma}_{ji}^k  \big).$$
\end{prop}

\begin{coro}
 Une connexion affine $\nabla$ étant donnée, il existe une connexion symétrique qui lui est projectivement équivalente
\end{coro}

\begin{defi}
 Une connexion affine  $\nabla$ est de {\it trace nulle relativement à un système de coordonnées} $x^1,\ldots,x^n$ si dans celui-ci, les coefficients de Christoffel $\Gamma_{ij}^k$ de $\nabla$  vérifient (pour 
$j=1,\ldots,n$)
\begin{equation}
\sum_{i=1}^n \Gamma_{ij}^i=
\sum_{i=1}^n \Gamma_{\!ji}^i=0 \,.
\end{equation}
\end{defi}
On prendra garde à ce que la  notion introduite dans la définition ci-dessus n'est pas invariante mais dépend du système de coordonnées. 

\'Etant donné une connexion projective $\Pi$ sur $U$, soit $\nabla$ une connexion affine symétrique qui la représente. Soient  $\Gamma^k_{ij}$ ses coefficients de Christoffel relativement à  des coordonnées  $x^1,\ldots,x^n$ fixées.  Posons alors 
\begin{equation}
\label{E:defPi}
 {\Pi}^k_{ij}=\Gamma^k_{\!ij}-\delta_{i}^k \Big( \frac{1}{n+1} 
\Gamma^\ell_{\ell j}\Big)  
-\delta_{j}^k \Big( \frac{1}{n+1} 
\Gamma^\ell_{i \ell }\Big)  
\end{equation}
pour tout $i,j,k=1,\ldots,n$. Les coefficients ${\Pi}^k_{ij}$ définissent 
une connexion affine sans torsion de trace nulle dans les coordonnées $x^1,\ldots,x^n$ qui est  (vu (\ref{E:defPi}) et d'après la \PP \ref{P:cnxprojoWEYL}) projectivement équivalente à $\nabla$. Par ailleurs, on montre facilement que, le système de coordonnées $x^1,\ldots,x^n$ étant toujours fixé, il existe une unique connexion affine symétrique de trace nulle dans la classe d'équivalence projective de $\nabla$. 
\begin{defi}
 Les  quantités $  {\Pi}^k_{ij}$ (pour $i,j,k=1,\ldots,n)$ sont les   {\it coefficients  de Thomas} de la connexion projective $\Pi$ relativement au système de coordonnées $x^1,\ldots,x^n$. 
\end{defi}

Soient alors $\overline{\Pi}^\gamma_{\alpha\beta}$ ($\alpha,\beta,\gamma=1,\ldots,n$) les coefficients de Thomas de $\Pi$ relativement à un autre système de coordonnées $\overline{x}^1,\ldots,\overline{x}^n$ sur $U$. On déduit sans difficulté des formules (\ref{E:transfochristofell}) que pour tout $\alpha,\beta,\gamma=1,\ldots,n$,~on~a
\begin{equation}
\label{E:transocoeffTHOMAS}
 \overline{\Pi}^\gamma_{\alpha\beta}={\Pi}^k_{ij}\frac{\partial x^i }{\partial \overline{x}^\alpha }\frac{\partial  x^j}{\partial \overline{x}^\beta}\frac{\partial 
\overline{x}^\gamma
}{\partial x^k}+ 
\frac{\partial^2 x^i }{\partial \overline{x}^\alpha \partial \overline{x}^\beta } \frac{\partial \overline{x}^\gamma}{\partial x^i}
- \frac{1}{n+1} \Big( 
\delta_\alpha^\gamma \frac{\partial \log \Delta }{ \partial \overline{x}^\beta} + 
\delta_\beta^\gamma \frac{\partial \log \Delta }{ \partial \overline{x}^\alpha}
\Big)
\end{equation} 
où  $\Delta$  désigne le déterminant jacobien du changement de paramétrisation 
$(\overline{x}^1,\ldots,\overline{x}^n)\mapsto (x^1,\ldots,x^n)$. 

Puisqu'une connexion projective est localement définie pas ses coefficients de Thomas, on en déduit la 
\begin{prop}
\label{P:connexionprojoTHOMAS}
 La donnée d'une connexion projective sur $M$ est équivalente à la donnée, pour chaque système de coordonnées $x^1,\ldots,x^n$ sur $M$, de coefficients ${\Pi}^k_{ij}$ vérifiant $ {\Pi}^k_{ij}={\Pi}^k_{ji}$ et 
$\sum_{\ell=1}^n 
{\Pi}^\ell_{\ell j}=0$ 
pour tout indice $i,j,k$  compris entre 1 et $n$ et satisfaisant aux lois de transformations 
{\rm (\ref{E:transocoeffTHOMAS})}. 
\end{prop}

\begin{rema} La proposition  ci-dessus permet de donner une définition plus formelle de ce qu'est une connexion projective.  Si $(U_\nu,x_\nu=(x_\nu^i,\ldots,x_\nu^n))$ est un atlas sur $M$ et si l'on note $\Pi_{\nu,ij}^k$ les coefficients de Thomas de $\Pi$ dans les cartes  $(U_\nu,x_\nu)$, alors les $\Pi_\nu=  \Pi_{\nu,ij}^k dx_\nu^i \otimes dx_\nu^i\otimes \frac{\partial}{\partial x_\nu^k}$ définissent une 0-cochaine sur $M$ à valeurs dans ${\rm Sym}^2(\Omega_M^1)\otimes T_M$. 
Avec ce formalisme, les formules de transformations (\ref{E:transocoeffTHOMAS}) signifient que le cobord  de $\{ \Pi_\nu \}$ 
est égal au 1-cocycle $ \{ S(x_{\mu\nu}) \}$ où $S(x_{\mu\nu})$ désigne la {\og dérivée de Schwarz\fg} du  changement de coordonnées $x_\nu\circ x_{\mu}^{-1}$ (voir \cite{gunning} et \cite{molzon}). La donnée d'une connexion projective sur $M$ est équivalente à la donnée d'une 0-cochaine à valeurs dans ${\rm Sym}^2(\Omega_M^1)\otimes T_M$  satisfaisant cette  propriété.
\end{rema}

\subsection{Interprétation en termes de connexion de Cartan}
Il nous semble intéressant d'interpréter en des termes plus conceptuels la notion de connexion projective. Cela pourra se révéler important par exemple lors de l'étude des tissus globaux sur les variétés compactes. 

Pour cela, nous nous appuirons sur  le concept d'{\og espace généralisé\fg} introduit par Cartan, qui peut être considéré comme un équivalent infinitésimal courbe de la notion de géométrie de Klein.  Le récent livre de Sharpe \cite{sharpe} a contribué à populariser ces notions et est une référence assez complète. Pour certains points cependant, nous ferons plutôt référence au livre \cite{kobayashi} de Kobayashi.

\subsubsection{Espaces généralisés et connexions de Cartan}
Soit $H$ un sous-groupe de Lie fermé d'un groupe de Lie complexe $G$  tel que l'espace homogène $G/H$ soit de dimension $n$. Une  géométrie de Klein modelée sur la paire $(G,H)$ sur une variété $M$ de dimension $n$ est la donnée d'un atlas $\{ (U_\nu,F_\nu) \}$ sur $M$, les $F_\nu$ étant à valeurs dans $G/H$ et vérifiant 
\begin{equation}
\label{E:GH}
 F_\mu=g_{\mu\nu}\cdot F_\nu 
\end{equation}
sur les intersections non-vides $U_\mu\cap U_\nu$, pour certains $g_{\mu\nu}\in G$ constants. 

Pour tout $\nu$, notons $P_\nu$ le $H$-fibré principal $P_\nu =F_\nu^* P$ sur $U_\nu$ obtenu en tirant en arrière par $F_\nu$ le fibré principal tautologique $P= G\rightarrow G/H$. Les relations (\ref{E:GH}) impliquent que les $P_\nu$ se recollent pour former un $H$-fibré principal à droite $P_M$ sur $M$. Soit  $\omega_G\in \Omega^1(G)\otimes   \mathfrak g$ la forme de Maurer-Cartan de $G$. Les ``tirés-en-arrière'' $F_\nu^*(\omega_G)$ se recollent eux aussi pour former une 1-forme différentielle 
$\omega\in \Omega^1(P_M)\otimes   \mathfrak g$ qui vérifie

\begin{enumerate}
\item[1.]  $\omega_{p}: T_p {P_M} \rightarrow \mathfrak g$ est un isomorphisme linéaire quel que soit $p\in P$;
 \item[2.] $\omega(X^\dagger)=X$ pour tout $X\in \mathfrak h$, où  $X^\dagger$ désigne le {\it champ de vecteurs fondamental} associé à $X$ (c'est-à-dire le générateur infinitésimal du (germe de) flot complexe $z\mapsto \exp(zX)$);
\item[3.]  $R_h^*(\omega)=Ad(h^{-1})\cdot \omega$ pour tout $h\in H$, où  $R_h $ désigne la multiplication à droite sur $P_M$ par $h$ et $Ad(h^{-1})$ dénote l'action adjointe de $h^{-1}$ sur $\mathfrak g$;
\item[4.] l'équation de Maurer-Cartan $d\omega+ \frac{1}{2} [ \omega,  \omega]\equiv 0$. 
\end{enumerate}

Une {\it géométrie de Cartan} (ou plus précisément une {\it connexion de Cartan})  sur $M$ modelée sur $(G,H)$ (ou plus justement sur $(\mathfrak g,\mathfrak h)$) est la donnée  d'un $H$-fibré principal à droite $P_M\rightarrow M$ et d'une 1-forme différentielle $\omega \in \Omega^1(P_M)\otimes   \mathfrak g$ qui vérifient les conditions 1., 2.  et 3.  ci-dessus, mais pas forcément la condition 4. Les géométries de Klein sont donc des cas particuliers de géométries de Cartan, que l'on dira {\it plates}. Une géométrie de Cartan localement isomorphe (en un sens naturel) à une géométrie de Klein est dite {\it intégrable}. En toute généralité, on définit  la {\it courbure de Cartan} de la géométrie de Cartan considérée  comme étant la 2-forme  différentielle $\mathfrak C(\omega)=d\omega+ \frac{1}{2} [ \omega, \omega]\in \Omega^2(P_M)\otimes \mathfrak g$. Celle-ci mesure le défaut de platitude de la géométrie de Cartan donnée,  
comme on le verra Section \ref{S:platitude}. 

La {\it torsion} $\mathfrak t(\omega)$ d'une connexion de Cartan est définie comme $\mathfrak t(\omega)=\rho( \mathfrak C(\omega))$ où l'on a noté $\rho:\mathfrak g \rightarrow \mathfrak g/\mathfrak h$ le passage au quotient canonique.  La connexion de Cartan $\omega$ est  {\it sans torsion} si $\mathfrak t(\omega)$ est identiquement nulle, c'est-à-dire si sa courbure $\mathfrak C(\omega)$ est à valeurs dans $\mathfrak h$.

L'action adjointe de $ G$ sur lui-même induit une action de $ H$ sur $ \mathfrak g$ et donc sur $ \mathfrak g/\mathfrak h$. 
 On en déduit une {\it action adjointe} de $H$ sur ${\rm Hom}( \wedge^2( \mathfrak g/\mathfrak h),\mathfrak g)$ qui fait de cet espace vectoriel complexe un $H$-module. La {\it fonction courbure} de $\omega$ est l'application $K: P\rightarrow {\rm Hom}( \wedge^2( \mathfrak g/\mathfrak h),\mathfrak g)$ définie de la façon suivante: si $p\in P_M$, alors 
$$ K_p(\zeta,\zeta'):=\mathfrak C(\omega)_p \Big(     \omega_p^{-1}(\zeta), \omega_p^{-1}(\zeta')  \Big)\,  $$
pour tout $\zeta,\zeta' \in  \mathfrak g/\mathfrak h$ (on montre que $K$ est bien définie et est un tenseur associé à $\omega$ ({\it cf.}     \cite[Lemma {\bf 5}.3.23]{sharpe})). 

Il est classique de s'intéresser aux géométries de Cartan {\og spéciales\fg}, {\it i.e.}  celles dont la fonction courbure est à valeurs dans un  sous-$H$-module spécifique que l'on décrète être le {\og sous-module  normal\fg} $N$ de ${\rm Hom}( \wedge^2( \mathfrak g/\mathfrak h),\mathfrak g)$. On peut donner une idée de l'intérêt de cette notion de la façon suivante: un bon nombre de structures géométriques\footnote{Par exemple: les connexions projectives, les structures conformes, les structures quasi-grassmanniennes, etc. Nous renvoyons le lecteur aux articles  \cite{ochiai,capschichl} pour un point de vue rigoureux sur cette question.}
peuvent être interprétées en termes de géométrie de Cartan (en général après plusieurs prolongations différentielles à des ordres plus grands). Se pose alors le problème d'associer de façon canonique une connexion de Cartan $\omega$  à la structure géométrique initiale. 
Dans bon nombre de cas, cela se fait de façon satisfaisante en demandant que $\omega$ soit normale (en ayant au préalable choisis et convenablement défini la ``bonne'' notion de normalité). Dans les cas où cette stratégie s'applique, cela force l'unicité de la connexion de Cartan compatible avec une $G$-structure donnée et permet (par exemple) d'aborder avec un cadre conceptuel bien dégagé les problèmes d'équivalence. 

Nous ne discuterons pas davantage la notion de {\og connexion de Cartan normale\fg} en toute généralité. Nous donnerons plus bas une définition précise  dans le  cas des connexions projectives, 
qui est le seul cas  qui nous~intéresse~ici.

\subsubsection{Jets, fibrés de repères et formes canoniques } 
Soit $m$ un point de $M$. 
Pour $k\geq 0$, on note $[\varphi]_0^k$ 
le jet à l'ordre $k$ en l'origine de $\varphi: (\mathbb C^n,0)\rightarrow (M,m)$. Si $\varphi $ est inversible, on dit que $[\varphi]_0^k$ 
est un repère d'ordre $k$ sur $M$ en $m$.  On note $R_k(M)$ l'ensemble des repères d'ordre $k$ sur $M$. On vérifie que de $M$, il hérite d'une structure de variété complexe faisant des applications de projections 
 $\pi_k^\ell: [\varphi]_0^k  \in R_k(M) \mapsto [\varphi]_0^\ell \in R_\ell(M) $ des submersions holomorphes (pour tout $k,\ell$ tels que $k \geq \ell$).

On note $G_k({\mathbb C}^n)$ l'espace des repères d'ordre $k$ en l'origine de $\mathbb C^n$: $ G_k({\mathbb C}^n)= \{ \,  [f]_0^k \, \big| \, f \in D( \mathbb C^n,0) \, \} $. 
La composition des (germes de) biholomorphismes induit une structure de groupe de Lie complexe sur $ G_k({\mathbb C}^n)$. 
On vérifie alors que l'action 
\begin{align*}
R_k(M) \times G_k(\mathbb C^n)& \longrightarrow R_k(M) \\
\big( [\varphi]_0^k, [f]_0^k       \big)
&\longmapsto [\varphi\circ f]_0^k
\end{align*}
fait de  $\pi_k: R_k(M) \rightarrow M$  un $ G_k({\mathbb C}^n)$-fibré principal (à droite).

Supposons que $k>0$. Chacun des fibrés  $ R_k(M)$ est muni d'une {\it 1-forme différentielle canonique} $\theta_k$ à valeurs dans l'espace tangent 
de $R_{k-1}(\mathbb C^n)$ en le jet à l'ordre $k-1$ de $Id_{\mathbb C^n}$,  que l'on note $I\!d$ quel que soit $k$.  Soit $v\in T_\xi\, R_k(M)$ avec $\xi=[\varphi]_0^k$. On définit $\theta_k(v)$ de la façon suivante: soit  $\overline{\xi}=\pi_k^{k-1}(\xi)=[\varphi]_0^{k-1}$.
Alors le  biholomorphisme local $\varphi: (\mathbb C^n,0)\rightarrow M$ induit un germe de biholomorphisme 
\begin{align*}
 \varphi^{[k-1]}:  \big( R_{k-1}(\mathbb C^n),I\!d\big) &\longrightarrow  \big(  R_{k-1}(M), \overline{\xi}\big)\\
 [g]_0^{k-1} \quad & \longmapsto \; \;[\varphi \circ g]_0^{k-1}.
\end{align*}
On vérifie que la différentielle de $\varphi^{[k-1]}$ en $I\!d$ ne dépend que de $\xi$ et pas du choix de $\varphi$. On la note 
\begin{align*}{\xi}_*=d \varphi^{[k-1]}_{I\!d} :  T_{\!I\!d}\, R_{k-1}(\mathbb C^n)
 \longrightarrow  T_{\overline{\xi}} \,R_{k-1}(M)\,.
\end{align*}
On pose alors 
\begin{align*}
 \theta_{k}(v)= \big( {\xi}_* \big)^{-1} \big( d \pi_k^{k-1}( v)  \big)\,.
\end{align*}

Les formes différentielles canoniques des fibrés de repères  satisfont plusieurs propriétés  d'invariance et de compatibilité ({\it cf.} \cite[\S 3]{kobayashiCAN}). Les deux propriétés des $\theta_k$ qui sont  le plus importantes ici nécessitent de faire quelques explications préliminaires. 
 Pour $k\geq 1$, on note ${\mathfrak g}_k(\mathbb C^n)$ l'algèbre de Lie (complexe) de $G_k(\mathbb C^n)$. 
L'action à droite de ce groupe de Lie sur $R_k(M)$ induit une application injective $X\mapsto X^  \dagger$ de  $\mathfrak g_k(\mathbb C^n)$ à valeurs dans l'algèbre des champs de vecteurs sur $R_k(M)$. D'autre part, la différentielle en l'identité de la projection {$G_k(\mathbb C^n)\rightarrow 
G_{k-1}(\mathbb C^n)$ induit un épimorhisme $\mathfrak g_{k}(\mathbb C^n)\ni X\mapsto X'\in \mathfrak g_{k-1}(\mathbb C^n)$.  Enfin, la projection $\pi_k^0:R_k(\mathbb C^n)\rightarrow R_0(\mathbb C^n)=\mathbb C^n$ induit une trivialisation  $T R_{k-1}(\mathbb C^n) = T \mathbb C^n \times \mathfrak g_{k-1}(\mathbb C^n)$, d'où on déduit une injection naturelle de $ \mathfrak g_{k-1}(\mathbb C^n)$ dans $  T_{\!I\!d} \, R_{k-1}(\mathbb C^n)$ . 
 Ces remarques faites, on vérifie alors que pour tout $X \in \mathfrak g_k(\mathbb C^n)$, on a 
\begin{align}
\label{E:thetakPROP1}
 \theta_k(X^  \dagger)\equiv X'.
\end{align}

Fixons maintenant $\Gamma \in D(\mathbb C^n,0)$. On peut lui associer le germe de biholomorphisme 
\begin{align*}
 \big( R_{k-1}(\mathbb C^n), I\!d \,  \big)& \longrightarrow  R_{k-1}(\mathbb C^n) \\
[\varphi]_0^{k-1}\; \,  & \longmapsto \big[\Gamma \circ \varphi \circ \Gamma^{-1}\big]_0^{k-1} 
\end{align*}
dont on vérifie que la différentielle en $I\!d$ ne dépend que de $\gamma=[\Gamma]_0^{k}\in G_k(\mathbb C^n)$. Notons la ${\rm Ad}(\gamma)$. Alors $ 
{\rm Ad}:\gamma \mapsto {\rm Ad}(\gamma)$ définit 
l'{\it action adjointe} (à droite) de  $G_{k}(\mathbb C^n)$ sur $T_{\!I\!d}\, R_{k-1}(\mathbb C^n)$.  Cette définition étant posée, on vérifie que pour tout  $\gamma \in G_k({\mathbb C}^n)$, on a 
\begin{align}
\label{E:thetakPROP2}
 R_\gamma^*(\theta_k)={\rm Ad}(\gamma^{-1})\cdot \theta_k \, .
 \end{align}

Le cas $k=2$ est particulièrement important et particulier. En effet, on a des identifications naturelles 
\begin{equation}
 \label{E:theta2}
 T_{\!I\!d}\, R_{1}(\mathbb C^n)\simeq \mathfrak{Aff}(\mathbb C^n) \simeq \mathbb C^n \oplus{\mathfrak{ gl}}_n(\mathbb C). 
\end{equation}
Il apparaît donc que la seconde forme canonique $ \theta_2$ est à valeurs dans une algèbre de Lie. Soient $(e_1,\ldots,e_n)$ et $(E_i^j)_{i,j=1}^n$ les bases canoniques de $\mathbb C^n$ et ${\mathfrak{ gl}}_n(\mathbb C)\simeq M_n(\mathbb C)$ respectivement. On peut alors décomposer $\theta_2$ composante par composante dans ces bases en tenant en compte de la somme directe (\ref{E:theta2}). On note 
$$ \theta_2=\sum_{i=1}^n \theta^ie_i+ \sum_{i,j=1}^n \theta^i_j E_i^j .$$

Si $\varphi: (\mathbb C^n,0)\rightarrow U \subset \mathbb C^n$ est un biholomorphisme, on a au second ordre 
$$ \varphi(x^1,\ldots,x^n) =\sum_{i=1}^n \Big(   
\varphi^i + \sum_{j=1}^n \varphi^i_j  x^j + \frac{1}{2} \sum_{j,k=1}^n \varphi^i_{jk} x^j x^k \Big)\,e_i $$ 
avec $(\varphi^i_j)_{i,j=1}^n$ inversible et où les $\varphi^i_{jk}$ sont symétriques en les indices $j$ et $k$. 
On en déduit que les $(\varphi^i, \varphi^i_j, \varphi^i_{jk})$  forment un système de coordonnées locales sur $R_2(U)$ dans lesquelles la fibration   $R_2(U)\rightarrow U$ n'est rien d'autre que la projection 
$ (\varphi^i, \varphi^i_j, \varphi^i_{jk}) \longmapsto (\varphi^i)$. On peut alors  exprimer (les composantes $\theta^i$ et $\theta^i_j$ de) la seconde forme canonique $\theta_2$  dans ces coordonnées. En notant $(\psi^i_j)_{i,j=1}^n$ l'inverse de $(\varphi^i_j)_{i,j=1}^n$,  on a ({\it cf.} \cite[p. 141]{kobayashi}): 
\begin{equation*}
\qquad\qquad
\theta^i=\sum_{j=1}^n \psi^i_j d \varphi^j \qquad \mbox{ et } \qquad 
\theta^i_j
=\sum_{k=1}^n \psi^i_k d \varphi^k_j+ \sum_{k,\ell,m=1}^n  \psi^i_k \varphi^k_{\ell j} \psi^\ell_m d \varphi^m.
\end{equation*}
Un calcul direct permet alors de montrer que, pour tout $i=1,\ldots,n$, la composante $\theta^i$ de $\theta_2$ vérifie  
\begin{equation}
\label{E:maurercartanformecan}
 d \theta^i=-\sum_{j=1}^n \theta_j^i \wedge  \theta^j \, . 
\end{equation}

\subsubsection{ Connexion projective à la Cartan} On note (avec un décalage d'indice) $e^0,\ldots,e^n$ la base canonique de $\mathbb C^{n+1}$ (ainsi $e^0=(1,0\ldots,0)$, $e^1=(0,1,0\ldots,0)$, etc). 
On a une décomposition en somme directe $\mathfrak g=  
\mathfrak g_{-1} \oplus 
\mathfrak g_0\oplus 
\mathfrak g_1$
de l'algèbre de Lie complexe 
$\mathfrak g= {\mathfrak sl}_{n+1}( \mathbb C)$ avec 
\begin{align*}
 \mathfrak g_{-1}  = 
& \, \left\{ 
\begin{pmatrix}
 0 & 0 \\
\xi & 0
\end{pmatrix} \; \Big| \; 
 \xi \in M_{n\times  1}(\mathbb C)     \right\}    \simeq   \, \mathbb C^n
\\
\mathfrak g_{0}  = & \, \left\{ 
\begin{pmatrix}
 a & 0 \\
0 & A
\end{pmatrix} \; \Big| 
\begin{tabular}{c}
 $A\in M_{n\times n}(\mathbb C)$\\
$a=-{\rm Tr}(A)  $
\end{tabular}\hspace{-0.15cm}
  \right\}    \simeq   \, 
\mathfrak{ gl}_n(\mathbb C) \\
\mbox{ et } \quad \mathfrak g_{1}  = 
& \, \left\{ 
\begin{pmatrix}
 0 & v \\
0 & 0
\end{pmatrix} \; \Big| \; 
v\in M_{1\times n}(\mathbb C)     \right\}    \simeq   \, \mathbb C^n.
\end{align*}
On pose également 
\[\mathfrak h=\mathfrak g_{0} \oplus 
\mathfrak g_1
=  \, \left\{ 
\begin{pmatrix}
 a & v \\
0 & A
\end{pmatrix} \; \Bigg| 
\begin{tabular}{c}
 $A\in M_{n\times n}(\mathbb C)$\\
$ v \in M_{1 \times n}(\mathbb C)  $\\
$    a=-{\rm Tr}(A)  $
\end{tabular}\hspace{-0.15cm}
  \right\}    \simeq   \, \mathbb C^n\oplus
\mathfrak{ gl}_n(\mathbb C) .
\]

C'est l'algèbre de Lie du sous-groupe de Lie $H$ de $ G=SL_{n+1}(\mathbb C)$ formé des éléments $g\in G$ laissant invariante la droite  
engendrée par $ e^0$.  Il en découle que le quotient $ G/H$  est l'espace projectif $\mathbb P^n(\mathbb C)$. Celui-ci est donc un espace homogène que l'on munit de la géométrie de Klein associée à la paire $(\mathfrak g, \mathfrak h)$. Par définition, une {\it connexion de Cartan projective} sur une variété $M$ est une connexion de Cartan  sur $M$ associée à une géométrie de Cartan modelée sur $(\mathfrak g, \mathfrak h)$.

Soit $\omega$ une connexion de Cartan projective. 
En tenant compte de la graduation  $\mathfrak{sl}_{n+1}(\mathbb C)
=\mathfrak g=\mathfrak g_{-1} \oplus \mathfrak g_0\oplus \mathfrak g_1$, on peut écrire matriciellement 
\[
\omega=
\begin{pmatrix}
 \omega_0 & \omega_i  \\
\omega^k & \omega_i^k
\end{pmatrix} 
 \]
où les 1-formes $\omega_0$, $\omega^k$, $\omega_i$ et $\omega^k_i$ (pour $i,k=1,\ldots,n$) vérifient $\omega_0=-\sum_{i=1}^n \omega^i_i=0$.  Alors on a 
\begin{equation}
\label{E:matriceOMEGA}
 \Omega= d\omega+\omega\wedge \omega = 
\begin{pmatrix}
 \Omega_0 & \Omega_i  \\
\Omega^k & \Omega_i^k
\end{pmatrix} =\begin{pmatrix}
 d\omega_0 + \omega_j \wedge \omega^j & d\omega_i  +\omega_0\wedge \omega_i +  \omega_j \wedge \omega_i^j\\
d\omega^k +\omega^k \wedge \omega_0+ \omega^k_j\wedge \omega^j
& d\omega_i^k+\omega^j\wedge \omega_j + \omega_i^j\wedge \omega_j^k
\end{pmatrix}.
\end{equation}
 Par définition, $\omega$ est sans torsion si  la 2-forme   $\rho(\Omega)$  à valeurs dans $\mathfrak g/\mathfrak h \simeq \mathfrak g_{-1}$ est identiquement nulle.  Cela équivaut à ce que $\Omega^k=0$ quel que soit $k=1,\ldots,n$.

 Il découle des propriétés d'invariance d'une connexion projecive  $\omega$ définie sur l'espace total $P$ d'un $H$-fibré principal au dessus de $U$ que celle-ci est complètement déterminée par sa restriction à n'importe quelle sous-variété $V\subset P$ de dimension $n$ transverse à la fibration $P\rightarrow U$. Cette remarque justifie l'astuce classique (due à Cartan)  consistant à utiliser une {\it jauge}, à savoir une section (locale) $ \sigma$ du fibré principal $P$, pour ramener l'étude de $\omega$ sur $P$ à celle de la {\it forme de soudure} $\omega_\sigma$ (associée à $\sigma$) sur $U$, qui est    définie comme étant le tiré-en-arrière $\omega_\sigma= \sigma^*(\omega) \in \Omega^1_U \otimes {\mathfrak{sl}_{n+1}(\mathbb C)}$. On impose aussi à $\sigma$ d'être telle  que la 1-forme $\overline{\omega}_\sigma=\rho( \omega_\sigma)\in 
\Omega^1_U \otimes (\mathfrak{g}/ \mathfrak{h})  $ 
induise un isomorphisme de $T_{U,u}$ sur $\mathfrak{g}/ \mathfrak{h}  $ quel que soit $u\in U$. 

Si $\widehat{\sigma}$  est une autre jauge, il existe une application $h:U\rightarrow H$ telle que $\sigma=\widehat{\sigma} \cdot h$. 
Par hypothèse, on a $R_k^*(\omega)=Ad(k^{-1})\cdot \omega$ pour tout $k\in H$ (puisque $\omega$ est une connexion de Cartan). On en déduit la relation 
\begin{equation}
\label{E:changementsoudure}
 \omega_{\hat{\sigma}}=Ad(h^{-1})\cdot \omega_\sigma+h^*( \omega_H)
\end{equation}
où l'on a noté $\omega_H$ la forme de Maurer-Cartan de $H$. 
\begin{prop}Les coordonnées $x^1,\ldots,x^n$ étant fixées, il existe essentiellement un unique choix de jauge $\sigma$ (que l'on dira être {\og la\fg}  {\it jauge standard} relativement aux $x^i$) telle que la forme de soudure associée soit de la forme 
\begin{equation}
\label{E:soudurestandard}
\omega_\sigma=
\begin{pmatrix}
 0 & \omega_i \\
dx^k &  \omega^k_i
\end{pmatrix} . 
\end{equation}
Le terme {\og  essentiellement \fg} signifie ici que deux jauges standards $\sigma_1$ et $\sigma_2$ sont liées par une relation de la forme 
$\sigma_2=\sigma_1 \cdot (\zeta\, {Id}_{n+1})$  où $\zeta$ désigne un  nombre complexe  tel que 
$\zeta^{n+1}=1$. 
\end{prop}

Notons  $\omega_i^k=\sum_{j=1}^n \omega_{ij}^k dx^j$  les décompositions en composantes scalaires des coefficients $\omega_{i}^k$ de la matrice (\ref{E:soudurestandard}) d'une forme de soudure standard
dans les coordonnées $x^1,\ldots,x^n$. Ils vérifient $\sum_{i=1}^n \omega_{ij}^i=0$ pour tout $j$. De (\ref{E:matriceOMEGA}) et 
(\ref{E:soudurestandard}) il vient que $\omega$ est sans torsion si et seulement si $\omega_{ij}^k=\omega^k_{ji}$ quels que soient $i,j,k=1,\ldots,n$.  

 On suppose désormais $\omega$ sans torsion dans tout ce qui suit.\vspace{0.15cm} 

Si  $\overline{\sigma}$ désigne ``la'' jauge standard relativement à un autre système de coordonnées $\overline{x}^1,\ldots,\overline{x}^n$  sur $U$, il existe une application $h:U\rightarrow H$ (essentiellement unique) telle que $\overline{\sigma}={\sigma}\cdot h$, que l'on peut expliciter ({\it cf.} \cite[\S3.2]{cs}). Utilisant (\ref{E:changementsoudure}), on peut alors établir le 
\begin{coro} Pour $\alpha,\beta,\gamma=1,\ldots,n$, 
 la composante scalaire $\overline{\omega}_{\alpha\beta}^\gamma$ (dans les coordonnées $\overline{x}^1,\ldots,\overline{x}^n$) du coefficient  $\overline{\omega}_{\alpha}^\gamma$ de ``la'' forme de soudure standard
$ \omega_{\overline{\sigma}}$ s'obtient à partir des composantes scalaires $\omega_{ij}^k$ selon les formules de transformations (\ref{E:transocoeffTHOMAS}). 
\end{coro}

On déduit des résultats précédents que les composantes scalaires $\omega_{ij}^k$ des coefficients  $\omega_{i}^k$ d'une forme de soudure standard d'une connexion de Cartan projective sans torsion peuvent être considérées comme des coefficients de Thomas. Vu la \PP \ref{P:connexionprojoTHOMAS}, il en découle qu'à une connexion de Cartan  projective sans torsion on peut associer de façon canonique une connexion projective au sens de la \DD \ref{D:defconnexionprojective}. \medskip

Tournons nous maintenant vers la question inverse: étant donnée une connexion projective au sens de la \DD \ref{D:defconnexionprojective}, peut-on lui associer  une connexion de Cartan  projective de façon canonique? 

La réponse est oui mais nécessite d'introduire la notion de  
connexion de Cartan  projective {\it normale}, ce que nous allons faire  maintenant en suivant \cite{sharpe} de près. Soit $\omega$ une connexion de Cartan projective sur un $H$-fibré principal $P\rightarrow M$ (on suppose toujours que $\omega$ est sans torsion). Alors sa courbure  peut être vue comme  un tenseur $\Omega: P\rightarrow {\rm Hom}\big(\wedge^2(\mathfrak g/\mathfrak h),\mathfrak h\big)$. 
Le passage au quotient $\mathfrak h\rightarrow \mathfrak h/\mathfrak g_{-1}\simeq \mathfrak g_0$ induit un morphisme (de $H$-module) 
\begin{align}
\mu_1: 
{\rm Hom}\big(\wedge^2(\mathfrak g/\mathfrak h),\mathfrak h\big)
\longrightarrow 
{\rm Hom}\big(\wedge^2(\mathfrak g/\mathfrak h),\mathfrak g_{0}\big)\,.
\end{align}
La représentation adjointe $ \mathfrak h\rightarrow  {\rm End}( \mathfrak g/\mathfrak h)$ induit un isomorphisme ${\rm ad}: \mathfrak h/\mathfrak g_{-1} \rightarrow {\rm End}( \mathfrak g/\mathfrak h)$ ({\it cf.} \cite[Lemma 8.1.7]{sharpe} d'où on déduit un isomorphisme 
\begin{align}
\mu_2: 
{\rm Hom}\big(\wedge^2(\mathfrak g/\mathfrak h),\mathfrak g_{0}\big)
\simeq 
\wedge^2(\mathfrak g/\mathfrak h)^* \otimes \mathfrak g_{0}
\stackrel{{\rm Id}\otimes {\rm ad}}{\longrightarrow}
\wedge^2(\mathfrak g/\mathfrak h)^* \otimes {\rm End}( \mathfrak g/\mathfrak h)\,.
\end{align}
On considère alors 
\begin{align}
\mu_3: 
\wedge^2(\mathfrak g/\mathfrak h)^* \otimes {\rm End}( \mathfrak g/\mathfrak h) \simeq     
\wedge^2(\mathfrak g/\mathfrak h)^* \otimes
(\mathfrak g/\mathfrak h)^*\otimes(\mathfrak g/\mathfrak h) & \; 
\longrightarrow \;  (\mathfrak g/\mathfrak h)^*\otimes(\mathfrak g/\mathfrak h)^* \nonumber \\
(A^*\wedge B^*)\otimes C^*\otimes \eta \; 
& \longmapsto \; ( B^*(\eta)A^*-A^*(\eta)B^*)\otimes  C^*\,.
\end{align}
On définit alors le {\it sous-$H$-module normal} $N$  de ${\rm Hom}\big(\wedge^2(\mathfrak g/\mathfrak h),\mathfrak h\big)$ comme étant 
$$ N={\rm Ker}\Big(
{\rm Hom}\big(\wedge^2(\mathfrak g/\mathfrak h),\mathfrak h\big) \stackrel{ \mu_3\circ \mu_2 \circ \mu_1  }{\longrightarrow}  (\mathfrak g/\mathfrak h)^*\otimes(\mathfrak g/\mathfrak h)^*\Big)\, . $$
Conformément à la {\og philosophie\fg} évoquée plus haut, on pose alors la
\begin{defi}
 Une connexion de Cartan projective sans torsion $\omega$  est {\it normale} si  sa courbure est à valeurs dans $N$.
\end{defi}
Il n'est pas difficile d'expliciter  en coordonnées cette notion de normalité.  Tout d'abord, on a le 
\begin{lemm}\cite[Prop. 2]{kobayashinagano}
\label{L:formcourbure}
La matrice de courbure (\ref{E:matriceOMEGA}) admet une décomposition de la forme 
\begin{equation}
 \Omega=\frac{1}{2}\sum_{u,v=1}^n K_{uv}\, \omega^u\wedge \omega^v
\end{equation}
où les $K_{uv}$ sont des applications définies sur $P$ et à valeurs dans $\mathfrak g$, antisymétriques en les indices $u$ et $v$.
\end{lemm}
\begin{rema}
 Ce résultat n'est pas spécifique aux connexions de Cartan projectives mais est vérifié par toute connexion de Cartan $\omega$ subordonnée à un espace généralisé modelé sur une paire $(\mathfrak l,  \mathfrak l')$ lorsque l'algèbre de Lie $\mathfrak l $ admet une décomposition en somme directe  $\mathfrak l=
\mathfrak l_{-1} \oplus \mathfrak l_0 \oplus  \cdots \oplus  \mathfrak l_k$ telle que 
$\mathfrak l'=\mathfrak l_0 \oplus  \cdots \oplus  \mathfrak l_k$.  Pour des indications concernant la preuve, voir \cite[p.137]{kobayashi} ou bien \cite[Proposition 2]{kobayashinagano} pour le cas projectif. 
\end{rema}
D'après le lemme ci-dessus, on peut noter   $\Omega_i^k=\frac{1}{2}\sum_{j=1}^n K_{iuv}^k \omega^u\wedge \omega^v$ (pour $i,k=1,\ldots,n$) la décomposition en composantes scalaires (antisymétriques en les indices inférieurs $u$ et $v$) du coefficient  $\Omega_i^k$ de la matrice de courbure (\ref{E:matriceOMEGA}). On vérifie alors  le 
\begin{lemm}
\label{L:normal}
  La connexion de Cartan sans torsion $\omega$  est normale si et seulement 
\[
 \Omega_0=0 \qquad \mbox{ et } \qquad \sum_{k=1}^n K^k_{ikv}=0
\quad \mbox{pour  tout } \,i,v=1,\ldots,n.
\]
\end{lemm}
Le lecteur intéressé  pourra consulter \cite[p. 336]{sharpe} où les calculs sont explicités. 
\medskip

On explique maintenant comment la notion de connexion projective telle qu'elle est définie dans la Section \ref{S:viaconnexionaffine} peut être interprétée comme une connexion de Cartan projective.  Le groupe des transformations projectives de $\mathbb P^n$ qui laissent fixe  un même point est isomorphe au sous-groupe $H(\mathbb C^n)$ de $D( \mathbb C^n,0)$ 
formé par les applications inversibles de la forme 
$$ \big(x^1,\ldots,x^n\big)\longmapsto  \left( \frac{ \sum_{i=1}^n s_i^1 \,x^i } { 1+\sum_{i=1}^n s_i\,x^i  }\, ,\,\ldots\,  , \, \frac{  \sum_{i=1}^n s_i^n\,x^i} {  1+\sum_{i=1}^n s_i\,x^i  }\right).$$
On vérifie que la restriction de $j_0^k: D( \mathbb C^n,0) \rightarrow G_k({\mathbb C}^n)$   à $H(\mathbb C^n)$ est injective si $k\geq 2$. On désigne par $H_k({\mathbb C}^n)$ son image (qui est donc isomorphe à $H(\mathbb C^n)$) et on note $\mathfrak h_k(\mathbb C^n)$ l'algèbre de Lie associée.

On peut identifier les connexions affines symétriques sur $M$ avec les $G$-structure du second ordre $B_G\subset R_2(M)$ avec $G=GL_n(\mathbb C)=G_1(\mathbb C^n)<G_2(\mathbb C^n)$ ({\it cf.} \cite{kobayashinagano}).  On vérifie que deux connexions affines $\nabla$ et $\nabla'$ sur $M$ sont projectivement équivalentes si et seulement si les deux $G$-structures associées $B_G(\nabla)$ et $B_G(\nabla')$ vérifient 
$$  H_2(\mathbb C^n)\cdot B_G(\nabla)=H_2(\mathbb C^n)\cdot B_G(\nabla').
$$
On montre ainsi ({\it cf.} \cite[Proposition 10]{kobayashinagano}) que sur la variété $M$, il existe une correspondance biunivoque entre les connexions projectives (au sens de la \DD \ref{D:defconnexionprojective})   et les $H^2(\mathbb C^n)$-structures du second ordre. 

Fixons à partir de maintenant une connexion projective $\Pi$ sur $M$. On note $B_\Pi\subset R_2(M)$ la $H_2(\mathbb C^n)$-structure associée. Soit  $\varpi$ la restriction de $\theta_2$ à $B_\Pi$ que l'on considère vu (\ref{E:theta2}) comme un élément 
$$ \varpi\in \Omega^2\big(   B_\Pi  \big) \otimes  \Big(
\mathfrak g_{-1} \oplus \mathfrak g_0
   \Big) 
 $$
où l'on a procédé aux identifications naturelles  $\mathfrak{g}_{-1}\simeq
\mathbb C^n
$ et  $\mathfrak g_0\simeq \mathfrak{gl}_n(\mathbb C) $.  Soit  $ \varpi=\varpi_{-1}+\varpi_0$ la décomposition de $\varpi$ induite par  l'identification $\mathfrak{Aff}(\mathbb C^n)=
\mathfrak g_{-1} \oplus {\mathfrak g}_0$. En utilisant les relations (\ref{E:thetakPROP1}), (\ref{E:thetakPROP2}) et (\ref{E:maurercartanformecan}) satisfaites par la seconde forme canonique, 
on peut alors montrer le 
\begin{theo}[\'E. Cartan \cite{cartan}]
\label{T:thmdeCARTAN}
 Il existe une unique 1-forme $\varpi_{1} \in \Omega^1(B_\Pi)\otimes \mathfrak g_{1}$ telle que
\begin{enumerate}
 \item $\omega_\Pi:= \varpi_{-1}+\varpi_0+\varpi_1$ constitue une connexion de Cartan projective sur le $H$-fibré principal $B_\Pi$;
\item la torsion $\mathfrak t(\omega_\Pi)$  de $\omega_\Pi$ est identiquement nulle;
\item $\omega_\Pi$ est une connexion de Cartan  projective normale.
\end{enumerate}
\end{theo}
\begin{proof}
 Reprenant la preuve du Théorème 4.2 de  \cite{kobayashi}, on montre qu'il y a existence et unicité locale. L'unicité locale permet alors d'en déduire le théorème par recollement. 
\end{proof}
\bigskip

On peut  résumer le contenu de cette sous-section par le 
\begin{coro}
 La notion de connexion projective introduite dans la \DD \ref{D:defconnexionprojective} co\"incide avec la notion de connexion de Cartan projective normale sans torsion.
\end{coro}

\subsection{Critère de platitude}
On rappelle ci-dessous la caractérisation des connexions projectives intégrables. Il s'agit d'un résultat classique qui remonte aux travaux \cite{lie,tresse} de Lie et Tresse dans le cas plan.  En dimension supérieure, la caractérisation des connexion projectives intégrables est due à Weyl \cite{weyl}  et à Cartan \cite{cartan} (indépendamment). 

\subsubsection{Caractérisation des connexions de Cartan intégrables.}
\label{S:platitude}
La notion  générale et abstraite de connexion de Cartan trouve l'une de ses justifications dans la simplicité et le caractère naturel du critère (très classique) qui   caractérise  l'intégrabilité. 
En effet, cette caractérisation   s'obtient facilement à partir des deux résultats d'existence  et  d'unicité suivants:
\begin{lemm} On suppose $M$ connexe et simplement connexe.
\begin{enumerate}
\item Soit $\varpi\in \Omega^1(M)\otimes \mathfrak g$. 
Il existe une application holomorphe $f:M\rightarrow G$ telle que $\varpi=f^*(\omega_G)$ si et seulement si $d\varpi+ \frac{1}{2} [ \varpi, \varpi ]=0$;
 \item Soient $f,\hat{f}:M\rightarrow G$ holomorphes. Il existe $g\in G$ tel que $\hat{f}=g\cdot f$ si et seulement si $f^*(\omega_G)=\hat{f}^*(\omega_G)$.
\end{enumerate}
\end{lemm}
\begin{proof}
 Nous renvoyons à \cite[\S1]{griffiths} ou encore à \cite[Chap. 3]{sharpe}.
\end{proof}

En combinant les  deux résultats précédents, on obtient immédiatement le théorème fondamental suivant: 
\begin{theo}
\label{T:caractrisationCARTANplat}
Une connexion de Cartan est intégrable si et seulement si sa courbure  de Cartan est identiquement nulle.
\end{theo}

\subsubsection{Intégrabilité d'une connexion projective.}

Nous allons expliciter la condition d'intégrabilité donnée par le 
\TT \ref{T:caractrisationCARTANplat} dans le cas des connexions projectives. On  note encore  $\omega$ la connexion de Cartan projective normale sans torsion  associée à une connexion projective  $\Pi$ sur $M$  via le \TT \ref{T:thmdeCARTAN}.

 Dans la section précédente, nous avons vu  que  dans un  système de coordonnées $x^1,\ldots,x^n$, la forme de connexion  standard  s'écrit matriciellement 
\begin{equation}
\label{E:matriceomega}
\begin{pmatrix}
 0 & \omega_i  \\
dx^k & \omega_i^k
\end{pmatrix} =\begin{pmatrix}
 0 &  \omega_{ij} dx^j\\
dx^k 
& \Pi_{ij}^kdx^j
\end{pmatrix},
\end{equation}
où les $\Pi_{ij}^k$ sont les coefficients de Thomas de $\Pi$. Puisque $\omega$ est normale et sans torsion,  
la courbure  $\Omega$ de  la forme de soudure standard  a une forme particulière.
Des lemmes \ref{L:formcourbure} et \ref{L:formcourbure}, on déduit 
 que l'on peut écrire  (\ref{E:matriceOMEGA}) sous la forme 
\begin{equation}
\label{E:matriceOMEGA2}
 \Omega= 
\begin{pmatrix}
 0 & \Omega_i  \\
0 & \Omega_i^k
\end{pmatrix} =\begin{pmatrix}
 0 &  \Omega_{iuv}\, dx^u\wedge dx^v\\
0
& \Omega_{iuv}^k \,dx^u\wedge dx^v
\end{pmatrix}.
\end{equation}
Nous allons exprimer les quantités  scalaires $\omega_{ij},\Omega_{ijuv}$ et $\Omega_ {ijuv}^k$ en termes des coefficients de Thomas de $\Pi$. 
\smallskip

 Quels que soient $i,j,k,\ell=1,\ldots,n$, on définit 
\begin{align}
\label{E:tenseurweyl}
W_{jk \ell}^i   =  
\Pi_{jk \ell}^i+ \frac{1}{n-1} \Big(  \delta_\ell^i \,\Pi_{jk} -\delta_k^i\,  \Pi_{j\ell}  \Big) 
\end{align}
où l'on a posé (avec sommation sur $m$ variant entre 1 et $n$)
\begin{align*}
 \Pi_{jk l}^i   =      
\frac{  \partial \Pi_{j\ell}^i  } { \partial x^k  } 
- 
\frac{  \partial \Pi_{jk}^i  } { \partial x^\ell  } 
+ \Pi_{j\ell}^m\,\Pi_{mk}^i
- \Pi_{jk}^m\,\Pi_{m\ell}^i
\qquad 
\mbox{ et } \qquad \Pi_{jk} = \Pi_{kj} =
\Pi_{j mk}^m =
 \frac{  \partial \Pi_{jk}^m  } { \partial x^m  } - 
 \Pi_{jm}^s\, \Pi_{sk}^m .
\end{align*}

On peut alors vérifier que (\ref{E:matriceomega}) et (\ref{E:matriceOMEGA2}) s'écrivent respectivement 
\begin{equation}
\label{E:matriceomegaII}
\begin{pmatrix}
 0 & \omega_i  \\
dx^k & \omega_i^k
\end{pmatrix} =\begin{pmatrix}
 0 &  \frac{1}{1-n} \Pi_{ij} \,dx^j\vspace{0.1cm}  \\
dx^k 
& \Pi_{ij}^k\,dx^j
\end{pmatrix} \qquad \mbox{ et }  \qquad 
\begin{pmatrix}
 0 & \Omega_i  \\
0 & \Omega_i^k
\end{pmatrix} =\begin{pmatrix}
 0 &  \frac{1}{1-n}\Pi_{iuv}\, dx^u\wedge dx^v \vspace{0.1cm}\\
0
& \frac{1}{2}\, W_{iuv}^k\, dx^u\wedge dx^v
\end{pmatrix}
\end{equation}
où pour $i,u,v=1,\ldots,n$ les  $\Pi_{iuv}$ désignent les coefficients antisymétriques en les indices $u$ et $v$, uniquement déterminés par les   relations suivantes 
\begin{equation*}
  d( \Pi_{ij} \,dx^j) +\Pi_{ku}\Pi_{iv}^k \,dx^u \wedge   dx^v  = \Pi_{iuv} \,dx^u\wedge dx^v . 
\end{equation*}

Explicitement, on a  (avec sommation sur $m$ variant entre 1 et $n$):
\begin{equation}
\label{E:Piiuvdef}
\Pi_{iuv}= \frac{1}{2} \Big(     
 \frac{  \partial \Pi_{iu} }{\partial x^v}-
\frac{  \partial \Pi_{iv}}{\partial x^u}
+\Pi_{iu}^m \Pi_{m v}- \Pi_{iv}^m \Pi_{m u}
\Big) . \end{equation}
On en déduit que les quantités $W_{jk \ell}^i $ 
sont les composantes d'un tenseur de type (1,3) associé à $\Pi$ de façon invariante. En effet, en utilisant 
 (\ref{E:transocoeffTHOMAS}) (ou bien (\ref{E:changementsoudure})), on vérifie qu'elles satisfont les lois de transformation 
\begin{align*}
\overline{W}^\alpha_{\beta\gamma\delta}= 
W_{jk l}^i   \frac{ \partial \overline{x}^\alpha    }{\partial {x}^i  } 
 \frac{ \partial  {x}^j    }{\partial  \overline{x}^\beta } 
 \frac{ \partial {x}^k   }{\partial \overline{x}^\gamma  } 
 \frac{ \partial {x}^l   }{\partial \overline{x}^\delta  } 
\end{align*}
quels que soient $\alpha,\beta,\gamma,\delta=1,\ldots,n$. 
\begin{defi}
Le tenseur 
\begin{equation}
\label{E:weyltenseur}
\mathfrak W_\Pi \in T _U\otimes (\Omega^1_U)^{\otimes 3}
\end{equation}
 dont les  composantes  sont données par les formules (\ref{E:tenseurweyl}) est le {\it tenseur de Weyl} de la connexion projective $\Pi$.
\end{defi}
Dans la situation regardée, le \TT \ref{T:caractrisationCARTANplat} s'énonce comme suit: 
\begin{coro}
\label{C:caractCNXIONprojoplate1}
 Une connexion projective $\Pi$ est intégrable si et seulement si sa courbure (\ref{E:matriceOMEGA2}) est identiquement nulle.  \end{coro}

 On  montre que les composantes 
$W_{jk l}^i$ du tenseur de  Weyl $\mathfrak W_\Pi$ vérifient les {\og identités de Bianchi-Cartan\fg} 
$W_{jk l}^i+W_{k lj}^i+W_{jjk }^i=0$ quels que soient $i,j,k,l$ compris entre 1 et $n$ . Cela implique en particulier que le tenseur de Weyl d'une connexion projective est toujours identiquement nul en dimension $n=2$. A contrario, lorsque $n>2$, on peut montrer que les $\Pi_{iuv}$ sont  identiquement nuls si tous les $W^i_{jk\ell}$ le sont.  On a donc le 

\begin{coro}
\label{C:caractCNXIONprojoplate2} En dimension $n>2$, 
 une connexion projective $\Pi$ est intégrable si et seulement si son  tenseur de Weyl $\mathfrak W_\Pi$ est identiquement nul. 
\end{coro}

Lorsque $n=2$, on déduit de l'annulation du tenseur de Weyl $\mathfrak W_\Pi$ que les quantités $\Pi_{iuv}$ satisfont aux lois de transformation
\begin{align*}
\overline{\Pi}_{\overline{\imath}\overline{u}\overline{v}}= \Pi_{iuv}
 \frac{ \partial  {x}^i    }{\partial  \overline{x}^{\overline{\imath}} } 
 \frac{ \partial {x}^u   }{\partial \overline{x}^{\overline{u}}} 
 \frac{ \partial {x}^v   }{\partial \overline{x}^{\overline{v} } } 
\end{align*}
quels que soient $ \overline{\imath},  \overline{u},\overline{v}  =1,\ldots,n$. Elles sont donc les composantes d'un  {\it tenseur de courbure} 3 fois covariant 
$$ \mathfrak P_\Pi \in  (\Omega^1_U)^{\otimes 3}$$
attaché à $\Pi$ de façon invariante.  On a alors le 
\begin{coro}
\label{C:caractCNXIONprojoplate3} En dimension $n=2$, 
 une connexion projective $\Pi$ est intégrable si et seulement si son  tenseur de courbure $\mathfrak P_\Pi$ est identiquement nul. 
\end{coro}

\subsection{Sous-variétés complètement géodésiques}
Nous allons maintenant faire le lien entre la théorie des connexions projectives et celle des systèmes différentiels du second ordre qui sont sous une  certaine forme normale bien spécifique. 

Commençons par rappeler ce qu'est une  sous-variété totalement géodésique. 
\begin{defi}
\label{D:defsousvarTG}
Une sous-variété $V\subset U$ est {\it totalement géodésique} pour 
 $\nabla$ si  pour  tout  $v\in V$ et pour toute direction tangente $ L\in \mathbb PT_{V,v}$, le germe en $v$ de la courbe géodésique issue de $L$ est inclus~dans~$V$. 
\end{defi}

\begin{exem}
 Les sous-espaces affines (de n'importe quelle dimension) sont des sous-variétés totalement géodésiques pour la connexion affine sur $\mathbb C^n$ dont les symboles de Chritofell sont identiquement nuls.
\end{exem}

\subsubsection{Systèmes différentiels des sous-variétés complètement géodésiques}

 On fixe  un entier $m$ compris entre $1$ et $n-1$. Soit $Z:t=(t^1,\ldots,t^m)\mapsto Z(t)=(Z^1(t),\ldots,Z^n(t))$ un germe d'application de rang $m$ qui paramétrise un germe de sous-variété $V^m\subset U$ de dimension $m$. Une traduction immédiate de la \DD 
\ref{D:defsousvarTG} ci-dessus donne que $V^m$ est totalement géodésique si et seulement si,  pour tout $t_0=(t_0^1,\ldots,t_0^m)$ dans le domaine de définition de $Z$ et pour tout $\tau=(\tau^1,\ldots, \tau^m)\in \mathbb C^m$ non nul, le germe d'application  $ s\in (\mathbb C,0) \mapsto   Z(t_0+s\, \tau) \in V^m $
 paramétrise une courbe géodésique.  La \PP \ref{E:eqnonparamGEOD}  combinée avec un simple calcul formel à l'ordre 2 nous donne alors la 
\begin{prop}
\label{P:systemTG}
 L'application $Z$ paramétrise une sous-variété $V^m$  totalement géodésique pour $\nabla$ si et seulement si il existe des fonctions scalaires $C^\ell: t\mapsto C^\ell(t)\in \mathbb C$ pour $\ell=1,\ldots,n$, telles qu'on ait identiquement 
\begin{equation}
 \label{E:systemvarTG}
\frac{\partial^2 Z^k}{\partial t^a\partial t^b}+ \sum_{i,j=1}^n \Gamma^k_{ij}  \frac{\partial Z^i}{\partial t^a} \frac{\partial Z^j}{\partial t^b}= C^\ell \frac{\partial Z^k}{\partial t^\ell}
\end{equation}
sur $U$, cela quels que soient les indices $k,a,b$ compris entre 1 et $n$. 
\end{prop}
 Une sous-variété générique $V^m$ de dimension $m$  vérifie  la condition de transversalité 
\begin{equation*}
(\mathcal T) \qquad\qquad\qquad
V^m \; \mbox{\it est transverse à la projection } (x^1,\ldots,x^n)\mapsto (x^1,\ldots,x^m).\qquad\qquad\qquad
\end{equation*}
Celle-ci  est vérifiée si et seulement si $V^m$ peut être  définie (localement) comme le graphe d'une application $t=(t^1,\ldots,t^m)\mapsto (Z^{m+1}(t),\ldots,Z^{n}(t))$ de rang $m$.  On suppose  que $V^m$ vérifie $(\mathcal T)$ dans~ce~qui~suit.  Sous cette hypothèse, on peut reformuler  la \PP \ref{P:systemTG} pour obtenir la 
\begin{prop}
\label{P:11}
 Les assertions  suivantes sont équivalentes: 
\begin{enumerate}
 \item $V^m$ est totalement géodésique pour la connexion affine $\nabla$;
\item il existe une paramétrisation $t=(t^1,\ldots,t^m)\mapsto \big(t^1,\ldots,t^m,Z^{m+1}(t),\ldots,Z^{n}(t)\big)$ de $V^m$ qui satisfait le  système  différentiel suivant (avec  $k=m+1,\ldots,n$ et $a,b=1,\ldots,m$ tels que $a\leq b$):
\begin{equation*}
\label{Sm}
(\mathcal S^m_\nabla) \quad : \quad 
\left\lbrace
\begin{array}{ccl}
\frac{\partial^2 Z^k}{\partial t^a\partial t^b}
 & = & \quad  \sum_{c=1}^m \sum _{i,j=m+1}^n 
\frac{\partial Z^i}{\partial t^a}
\frac{\partial Z^j}{\partial t^b}
\frac{\partial Z^k}{\partial t^c} \Gamma^c_{ij}  \vspace{0.2cm}\\
 & &
  +\sum_{c=1}^m \sum _{i=m+1}^n 
\frac{\partial Z^k}{\partial t^c}
\left( 
\frac{\partial Z^i}{\partial t^a}\Gamma^c_{ib}+
\frac{\partial Z^i}{\partial t^b}\Gamma^c_{ai}
\right)- \sum _{i,j=m+1}^n \frac{\partial Z^i}{\partial t^a}
\frac{\partial Z^j}{\partial t^b} 
\Gamma^k_{ij}\vspace{0.2cm} \\
 & &  + 
\sum_{c=1}^m \frac{\partial Z^k}{\partial t^c}\Gamma^c_{ab}
-\sum _{i=m+1}^n  \frac{\partial Z^i}{\partial t^a}\Gamma^k_{ib}
-\sum _{i=m+1}^n \frac{\partial Z^i}{\partial t^b}\Gamma^k_{ai}  - \Gamma^k_{ab}\, ;
\end{array} 
 \right.\vspace{0.2cm} 
\end{equation*}
\item toute paramétrisation   $t
\mapsto \big(t^1,\ldots,t^m,Z^{m+1}(t),\ldots,Z^{n}(t)\big)$ de $V^m$ satisfait  $(\mathcal S^m_\nabla)$.
\end{enumerate}
\end{prop}
Pour  une démonstration de cette proposition  via la méthode du repère mobile, on peut consulter \cite[\S40-41]{hachtroudi}. 
\begin{rema}
 Le système $(\mathcal S^1_\nabla) $ n'est rien d'autre que le système  (\ref{E:eqnonparamGEOD}) de la \PP \ref{P:systemdiffCgeodesique}  caractérisant les courbes géodésiques pour $\nabla$. Il est complètement intégrable pour une raison évidente. Par contre, aucun des systèmes $(\mathcal S^m_\nabla) $ pour $m>1$ n'est complètement intégrable a priori. On peut montrer ({\it cf.} \cite{hachtroudi,yen}) que si l'un deux l'est, alors $\nabla$ est {\it projectivement plate}, {\it i.e} $[\nabla]$ est intégrable (et donc  les systèmes différentiels $(\mathcal S^m_\nabla) $ sont tous complètement intégrables). 
\end{rema}

La proposition ci-dessus montre que, étant donné un système de coordonnées $x^1,\ldots,x^n$ sur $U$, à toute connexion affine $\nabla$ on peut associer  la série $(\mathcal S^m_\nabla) $ de systèmes différentiels du second ordre dont les solutions paramétrisent des sous-variétés totalement géodésiques pour $\nabla$. Il découle immédiatement des définitions  \ref{D:defequivprojo} et \ref{D:defsousvarTG}  que    deux connexions affines sont projectivement équivalentes si et seulement si elles ont les mêmes sous-variétés totalement géodésiques.  On peut donc poser la
\begin{defi}
\label{D:defsousvarTGprojo}
Une  sous-variété $V$ est {\it totalement géodésique} pour 
une connexion projective  $\Pi$ si  $V$ l'est pour une (et alors pour toute) connexion affine qui représente  $\Pi$ localement. 
\end{defi}

Cette définition étant posée et justifiée, on 
peut formuler une version projective de la \PP \ref{P:11} (et qui s'en déduit de façon immédiate).  Notons $\Pi_{ij}^k$ les coefficients de Thomas d'une connexion projective $\Pi$ sur $U$ relativement à un système de coordonnées $x^1,\ldots,x^n$. Soit $V^m$ une sous-variété de $U$ de dimension $m<n$  vérifiant la condition de transversalité ($\mathcal T$).
\begin{prop} 
\label{P:11projo}
Les assertions suivantes sont équivalentes:
\begin{enumerate}
\label{P:Spi}
 \item $V^m$ est totalement géodésique pour la connexion projective $\Pi$;
\item il existe une paramétrisation $t=(t^1,\ldots,t^m)\mapsto \big(t^1,\ldots,t^m,Z^{m+1}(t),\ldots,Z^{n}(t)\big)$ de $V^m$ qui satisfait le  système  différentiel suivant (avec  $k=m+1,\ldots,n$ et $a,b=1,\ldots,m$ tels que $a\leq b$):
\begin{equation*}
\label{Sm}
(\mathcal S^m_\Pi) \quad : \quad 
\left\lbrace
\begin{array}{ccl}
\frac{\partial^2 Z^k}{\partial t^a\partial t^b}
 & = & \quad  \sum_{c=1}^m \sum _{i,j=m+1}^n 
\frac{\partial Z^i}{\partial t^a}
\frac{\partial Z^j}{\partial t^b}
\frac{\partial Z^k}{\partial t^c} \Pi^c_{ij}  \vspace{0.2cm}\\
 & &
  +\sum_{c=1}^m \sum _{i=m+1}^n 
\frac{\partial Z^k}{\partial t^c}
\left( 
\frac{\partial Z^i}{\partial t^a}\Pi^c_{ib}+
\frac{\partial Z^i}{\partial t^b}\Pi^c_{ai}
\right)- \sum _{i,j=m+1}^n \frac{\partial Z^i}{\partial t^a}
\frac{\partial Z^j}{\partial t^b} 
\Pi^k_{ij}\vspace{0.2cm} \\
 & &  + 
\sum_{c=1}^m \frac{\partial Z^k}{\partial t^c}\Pi^c_{ab}
-\sum _{i=m+1}^n  \frac{\partial Z^i}{\partial t^a}\Pi^k_{ib}
-\sum _{i=m+1}^n \frac{\partial Z^i}{\partial t^b}\Pi^k_{ai}  - \Pi^k_{ab}\, ;
\end{array} 
 \right.\vspace{0.2cm} 
\end{equation*}
\item toute paramétrisation   $t
\mapsto \big(t^1,\ldots,t^m,Z^{m+1}(t),\ldots,Z^{n}(t)\big)$ de $V^m$ satisfait  $(\mathcal S^m_\Pi)$.
\end{enumerate}
\end{prop}

\begin{rema}
 En dimension $n=2$, il n'y a qu'un seul système $(\mathcal S^m_\Pi)$ associé à une connexion projective $\Pi$. C'est le système 
$(\mathcal S^1_\Pi)$ qui  se réduit en fait à l'équation différentielle 
\begin{equation}
 \label{E:y''}
\frac{d^2y}{dt^2}= A \big(\frac{dy}{dt}\big)^3+B \big(\frac{dy}{dt}\big)^2+C \frac{dy}{dt}+D
\end{equation}
où $A,B,C$ et $D$ s'expriment en fonction des coefficients de Thomas de $\Pi$ via les relations 
\begin{equation}
\label{E:coeffy''}
  A=\Pi^1_{22}, \quad B=2\,\Pi^1_{12}- \Pi^2_{22}
, \quad C=\Pi^1_{11} -2\,\Pi^2_{12} \quad \mbox{ et } \quad 
D=-\Pi_{11}^2\,. 
\end{equation}

En combinant (\ref{E:coeffy''}) avec les relations  $\Pi^k_{ij}=\Pi^k_{ji}$ et $\Pi_{1j}^1+\Pi_{2j}^2=0$, on peut exprimer les $\Pi_{ij}^k$ en fonctions de $A,B,C$ et $D$. On en déduit des formules explicites  exprimant $\Pi_{112}$ et $\Pi_{212}$ (définis en (\ref{E:Piiuvdef})) en fonction de $A,B,C,D$ et de leurs dérivées partielles d'ordre 1 et 2. D'après le \CC \ref{C:caractCNXIONprojoplate3}, $\Pi$ est intégrable si et seulement si $\Pi_{112}=\Pi_{212}=0$. Lorsqu'on exprime ces deux conditions en termes de $A,B,C,D$ et de leurs dérivées partielles, on retrouve le critère de Lie et Tresse caractérisant les équations (\ref{E:y''}) équivalentes par un changement de coordonnées ponctuelles, à l'équation plate  ${d^2\overline{y}}/{d\overline{t}^2}= 0$. \end{rema}

\subsubsection{Restriction d'une connexion projective}
\label{S:restriction}
On commence par rappeler deux faits élémentaires. Soit  $V$ une sous-variété de $M$ munie des deux  connexions affines $\nabla$ et $\overline{\nabla}$. Alors:
(1) si  $V$ est totalement géodésique pour $\nabla$, celle-ci induit de façon naturelle  une connexion affine $\nabla|_V$  sur $V$ (par restriction); (2) si  $V$ est totalement géodésique pour $\nabla$ et pour $\overline{\nabla}$,  les connexions affines induites $\nabla|_V$ et $\overline{\nabla}|_V$
sont projectivement équivalentes (en tant que connexions affines sur $V$) si $\nabla$ et $\overline{\nabla}$ le sont. 

On peut donc poser  la 
\begin{defi}
 Soit $V$ une sous-variété totalement géodésique pour une connexion projective $\Pi$ sur $M$. Par restriction, $\Pi$ induit une connexion projective  sur $V$   bien définie, que l'on note $\Pi|_V$.
\end{defi}

On vérifie alors sans difficulté que l'opération de restriction d'une connexion projective le long d'une sous-variété totalement géodésique satisfait les propriétés énoncées par le 
\begin{lemm}
\label{L:projint}
Soient $V$ et $V'$ deux  sous-variétés totalement géodésiques pour une connexion projective $\Pi$ sur $M$. Alors 
\begin{enumerate}
 \item la connexion projective induite  $\Pi|_V$ est intégrable si $\Pi$ l'est.
\item  sous l'hypothèse que l'intersection $V\cap V'$ est lisse, celle-ci est encore une sous-variété totalement géodésique pour $\Pi$. De plus 
$  \big(\Pi|_{V}     \big)|_{V'}= \big(\Pi|_{V'}     \big)|_{V}=\Pi|_{V\cap V'} $.
\end{enumerate}
\end{lemm}

Bien qu' élémentaire, ce lemme est particulièrement utile pour prouver la \PP \ref{P:existePI}. Celle-ci  est   cruciale pour établir le \TT \ref{T:main} qui est notre résultat principal.

\section{Critère de linéarisabilité d'un tissu en hypersurfaces}
\label{S:codim1}
Nous abordons maintenant la question centrale examinées dans cet article: celle de reconnaître de façon effective si un tissu donné est {\it linéarisable}

\subsection{Linéarisabilité  des $(n+2)$-tissus en dimension $n$}
Dans cette section, on fixe et on note $W=(\mathcal F^1,\ldots, \mathcal F^{d})$ un $d$-tissu de codimension 1 sur $U$. Pour tout $i=1,\ldots,d$, l'on désigne par $\Omega^i$ une 1-forme différentielle sur $U$ qui définit $\mathcal F^i$. 
\begin{defi}
 Une connexion projective $\Pi$ est  {\it compatible} avec $W$ si les feuilles des  feuilletages qui constituent $W$ sont des hypersurfaces totalement géodésiques pour  $\Pi$. 
\end{defi}
\begin{exem}
\label{Ex:platelin}
 La connexion projective plate est compatible avec n'importe quel tissu linéaire.
\end{exem}

\begin{prop}
\label{P:existePI}
 Il existe une unique connexion projective compatible avec  un $(n+2)$-tissu donné. 
\end{prop}
\begin{proof}
 L'assertion que l'on veut démontrer étant de nature invariante, on ne perd pas en généralité en supposant que  $n$ des feuilletages de $W$ sont normalisés. On  fixe donc un système de coordonnées $x^1,\ldots,x^n$ sur $U$ dans lesquelles on a 
\begin{equation}
 \Omega^m= 
 \left\lbrace
                                         \begin{array}{rcl}
                                               dx^m-dx^n & \text{ pour}  & m = 1,\ldots,n-1\,, \vspace{0.08cm}\\
                                                dx^n & \text{ pour} & m=n\,, \vspace{0.08cm}\\
\sum_{a=1}^{n-1} \Omega^m_{a}\, dx^i -dx^n  &
\text{ pour} & m=n+1, n+2\,.\vspace{0.08cm}
                                              \end{array}
                                            \right.
\end{equation}
Ces normalisations impliquent que pour $m=1,\ldots,n+2$, une feuille de  $\mathcal F^m$ satisfait la condition 
($\mathcal T$) de la Section \ref{S:viaconnexionaffine} et donc 
est paramétrisée par une application 
$ \mathcal Z^m: t=(t^1,\ldots,t^{n-1})\mapsto (t,Z^m(t))$  où $Z^m$ vérifie 
\begin{equation}
 \label{E:eqdiffZm}
\frac{\partial Z^m}{\partial t^a}  = 
 \left\lbrace
                                         \begin{array}{lcl}
                                               \qquad  \delta^m_a  & \text{pour}  & m = 1,\ldots,n\,, \vspace{0.08cm}\\
                                                 \Omega^m_{a}\big(t,\mathcal Z^m(t)\big) & \text{pour} & m=n+1, n+2\, ,
                                              \end{array}
                                            \right.
\end{equation}
quel que soit $a$ compris entre $1$ et $n-1$. 

La prolongation différentielle au second ordre des équations  (\ref{E:eqdiffZm}) est immédiate. On obtient
\begin{equation}
 \label{E:eqdiffZm2}
\frac{\partial^2 Z^m}{\partial t^a\partial t^b }  = 
 \left\lbrace
                                         \begin{array}{lcl}
                                                 \qquad 0  & \text{si}  & m = 1,\ldots,n\,, \vspace{0.08cm}\\
                                                  X^m_{a}\big(  \Omega^m_{b}   \big)    & \text{si} & m=n+1, n+2\, ,
                                              \end{array}
                                            \right.
\end{equation}
pour tout $a,b$ tels que $1 \leq a\leq b<n$, où l'on a posé 
$X^m_{a}=\frac{\partial }{\partial x^a}+\Omega_{a}^m\frac{\partial }{\partial x^n}$ quel que soit $m$. 

Soit $\Pi_{ij}^k$ les coefficients de Thomas (dans les coordonnées $x^1,\ldots,x^n$) d'une connexion projective $\Pi$ sur $U$. Notons 
\begin{align*}
S(\Pi)_{ab}\, : \quad  
\frac{\partial^2 Z}{\partial t^a\partial t^b}
=\,& 
\frac{\partial Z}{\partial t^a}\frac{\partial Z}{\partial t^b}
 \sum_{c=1}^{n-1} A^c \frac{\partial Z}{\partial t^c}
+ \sum_{c=1}^{n-1} \Big(\frac{\partial Z}{\partial t^a}\,B_b^c+\frac{\partial Z}{\partial t^b} \,B_a^c\Big) \frac{\partial Z}{\partial t^c}   +  \sum_{c=1}^{n-1} C_{ab}^c\frac{\partial Z}{\partial t^c}  + D_{ab} 
\end{align*}
les équations différentielles scalaires du second ordre (pour $a,b$ tels que 
$1 \leq a\leq b \leq n-1$) qui constituent le système $(S_\Pi)$ des hypersurfaces totalement géodésiques pour $\Pi$. 
La \PP \ref{P:Spi} nous dit  que $\Pi$ est admissible par $W$ si et seulement si  les fonctions $Z_1,\ldots,Z_{n+2}$ satisfont $S(\Pi)$.  En d'autres termes:  la proposition  est démontrée si l'on établit l'existence et l'unicité de coefficients $A^c,B^c_a,C^c_{ab}$ et $D_{ab}$ holomorphes sur $U$ (pour  $a,b,c=1,\ldots,n-1$), symétriques en les indices $a$ et $b$, tels que $Z_m$ soit solution de toutes les équations $S(\Pi)_{ab}$, cela quel que soit $m\in \{1,\ldots,n+2\}$. Vu  (\ref{E:eqdiffZm}) et (\ref{E:eqdiffZm2}), cela équivaut à ce que soient vérifiées les relations 
\begin{align}
\label{E:ee}
0=\,& 
\delta_a^m\delta_b^m
 \sum_{c=1}^{n-1} A^c \delta_c^m 
+ \sum_{c=1}^{n-1} \Big(\delta^a_m\,B_b^c+\delta^b_m \,B_a^c\Big) 
\delta^c_m   +  \sum_{c=1}^{n-1} C_{ab}^c
\delta^c_m  + D_{ab} 
 \\
\mbox{et } \qquad  X_{ a}^\ell\big(  \Omega_{ b}^ \ell  \big) 
=\,& 
 \Omega_{ a}^ \ell   \Omega_{ b}   ^\ell
 \sum_{c=1}^{n-1} A^c \Omega_{ c}  ^ \ell
+ \sum_{c=1}^{n-1} \Big( \Omega_{ a}  ^\ell 
\,B_b^c+ \Omega_{ b}  ^ \ell
\,B_a^c\Big)
 \Omega_{ c} ^  \ell
   +  \sum_{c=1}^{n-1} C_{ab}^c
 \Omega_{ c}  ^ \ell
 + D_{ab}  \nonumber 
\end{align}
pour $m=1,\ldots,n $, $\ell \in \{n+1,n+2\}$ et $a,b$ tels que $1 \leq a\leq b \leq n-1$. \smallskip

Un simple décompte montre que les relations (\ref{E:ee}) induisent $n(n-1)(n+2)/2$ 
équations linéaires  non-homogènes en les coefficients $A^c,B^c_a,C^c_{ab}$ et $D_{ab}$.
 Or ceux-ci sont également au nombre de $n(n-1)(n+2)/2$ (en prenant en compte leurs symétries par rapport aux indices $a$ et $b$). Une conséquence immédiate  est qu'il suffit de montrer l'unicité des tels coefficients: l'existence sera alors assurée par un simple fait d'algèbre linéaire.  Mieux encore, il suffit de montrer que seuls les coefficients nuls satisfont (\ref{E:ee}) lorsqu'on suppose que les membres de gauche $X_{a}^\ell (  \Omega_{b}^\ell   )$ de ces équations sont  identiquement nuls pour $\ell=n+1$ et $\ell=n+2$. 

On suppose donc que les coefficients $A^c,B^c_a,C^c_{ab}$ et $D_{ab}$ vérifient 
(pour $m=1,\ldots,n $ et $\ell =n+1,n+2$) 
\begin{align}
\label{E:eee}
0=\,& 
\delta_a^m\delta_b^m
 \sum_{c=1}^{n-1} A^c \delta_c^m 
+ \sum_{c=1}^{n-1} \Big(\delta_a^m\,B_b^c+\delta_b^m \,B_a^c\Big) 
\delta_c^m   +  \sum_{c=1}^{n-1} C_{ab}^c
\delta_c^m  + D_{ab} 
 \\
\mbox{et } \qquad  0 
=\,& 
 \Omega_{ a}^\ell    \Omega_{ b}^ \ell  
 \sum_{c=1}^{n-1} A^c \Omega_{ c}^\ell   
+ \sum_{c=1}^{n-1} \Big( \Omega_{ a}^ \ell  
\,B_b^c+ \Omega_{ b}^\ell   
\,B_a^c\Big)
 \Omega_{\ell c}   
   +  \sum_{c=1}^{n-1} C_{ab}^c
 \Omega_{ c}^  \ell 
 + D_{ab}.  \nonumber 
\end{align}
 Montrons alors qu'on a forcément  $A^c=B^c_a=C^c_{ab}=D_{ab}=0$ quels que soient  $a,b,c$   entre $1$ et $n-1$ avec $a$ plus petit que $b$. Pour ce faire nous raisonnerons par récurrence sur la dimension $n$ de l'espace ambiant. \smallskip

Le cas $n=2$ est facile: la question se ramène à celle de la non-annulation d'un certain déterminant de Vandermonde. C'est un cas classique  traité dans \cite[\S 29]{bb}, référence à laquelle nous renvoyons le lecteur. 
\smallskip

Le cas suivant $n=3$ peut aussi être traité par un calcul direct que l'on va  détailler par souci de complétude et surtout parcequ'il est important pour la suite. On peut écrire les équations (\ref{E:eee}) sous la forme  matricielle
 \begin{equation} 
 \label{E:n=3EQUATION} 
 M_W  S=0 \end{equation}
avec 
 $${}^tS=\big( A^1,A^2,B_1^1,B_1^2,B_2^1,B_2^2,C_{11}^1,C_{12}^1,C_{22}^1,
C_{11}^2,C_{12}^2,C_{22}^2,D_{11},D_{12},D_{22}
\big)
$$
et 
\begin{equation*}
M_W=
\left(
 \begin{array}{ccccccccccccccc}
1& 0& 2& 0& 0& 0& 1& 0& 0& 0& 0& 0& 1& 0& 0\\ 
0& 0& 0& 0& 1& 0& 0& 1& 0& 0& 0& 0& 0& 1& 0\\
0& 0& 0& 0& 0& 0& 0& 0& 1& 0& 0& 0& 0& 0& 1\\
0& 0& 0& 0& 0& 0& 0& 0& 0& 1& 0& 0& 1& 0& 0\\
0& 0& 0& 1& 0& 0& 0& 0& 0& 0& 1& 0& 0& 1& 0\\ 
0& 1& 0& 0& 0& 2& 0& 0& 0& 0& 0& 1& 0& 0& 1\\
0& 0& 0& 0& 0& 0& 0& 0& 0& 0& 0& 0& 1& 0& 0\\ 
0& 0& 0& 0& 0& 0& 0& 0& 0& 0& 0& 0& 0& 1& 0\\ 
0& 0& 0& 0& 0& 0& 0& 0& 0& 0& 0& 0& 0& 0& 1\\ 
(\Omega^4_1)^3&  (\Omega^4_1)^2\Omega^4_2& 2\, (\Omega^4_1)^2 &  2\,\Omega^4_1\,\Omega^4_2& 
 0&  0&  \Omega^4_1& 
 0 &  0 &  0 &  0 & \Omega^1_2 & 1&  0&  0 \\
(\Omega^4_1)^2\,(\Omega^4_2)& \Omega^4_1\,(\Omega^4_2)^2&
 \Omega^4_1\,\Omega^4_2& (\Omega^4_2)^2& (\Omega^4_1)^2&\Omega^4_1\,\Omega^4_2 & 0& \Omega^4_1& 0& 0& \Omega^4_2& 0&0&1&0\\
\Omega^4_1(\Omega^4_2)^2&
 (\Omega^4_2)^3&
 0&
 0&
 2\,\Omega^4_1\,\Omega^4_2&
2\,(\Omega^4_2)^2
&
 0&
 0& 
\Omega^4_1
&
 0&
 0&
\Omega^4_2&
 0&
 0&
 1   \\
(\Omega^5_1)^3&  (\Omega^5_1)^2\Omega^5_2& 2\, (\Omega^5_1)^2 &  2\,\Omega^5_1\,\Omega^5_2& 
 0&  0&  \Omega^5_1& 
 0 &  0 &  0 &  0 & \Omega^5_2 & 1&  0&  0 \\
(\Omega^5_1)^2\,(\Omega^5_2)& \Omega^5_1\,(\Omega^5_2)^2&
 \Omega^5_1\,\Omega^5_2& (\Omega^5_2)^2& (\Omega^5_1)^2&\Omega^5_1\,\Omega^5_2 & 0& \Omega^5_1& 0& 0& \Omega^5_2& 0&0&1&0\\
\Omega^5_1(\Omega^5_2)^2&
 (\Omega^5_2)^3&
 0&
 0&
 2\,\Omega^5_1\,\Omega^5_2&
2\,(\Omega^5_2)^2
&
 0&
 0& 
\Omega^5_1
&
 0&
 0&
\Omega^5_2&
 0&
 0&
 1   \\
\end{array}\right).
\end{equation*}\smallskip

Par un calcul direct, on montre que 
\begin{equation}
 \label{E:n=3MATRICE}
{\rm det}\big( M_W\big) = 4 \!\!\prod_{ i<j<k} \Big( 
\frac{\Omega^i\wedge \Omega^j\wedge \Omega^k
}{dx^1\wedge dx^2\wedge dx^3} 
\Big)\,.
\end{equation}

L'hypothèse de position générale satisfaite par les feuilletages de $W$ se traduit par le fait que le terme de droite de l'égalité  (\ref{E:n=3MATRICE}) est un élément de ${\mathcal O}_U$ inversible. Par conséquent $M_W \in GL_{15}({\mathcal O}_U)$ et donc (\ref{E:n=3EQUATION}) n'admet que la solution triviale comme solution. Cela démontre la proposition dans le cas $n=3$. \smallskip

Considérons maintenant le cas général $n>3$ que l'on va traiter en se ramenant au cas $n=3$ par tranchage. 
Fixons $\alpha$ et $\beta$  tels que $1\leq \alpha< \beta <n$ et notons $V_{\alpha\beta}$ l'intersection de feuilles des feuilletages $\mathcal F_i$ pour $i\in \{1,\ldots,n-1\}\setminus \{\alpha,\beta\}$. C'est l'intersection d'un 3-plan affine avec $U$, sur laquelle $x^\alpha,x^\beta$ et $x^n$ induisent un système de coordonnées affines.  L'hypothèse de position générale satisfaite par les feuilletages de $W$ permet de considérer la trace $W_{\!{\alpha\beta}}=W|_{V_{\alpha\beta}}$ de $W$ sur $V_{\alpha\beta}$: c'est le 5-tissu induit par les 1-formes différentielles 
\begin{equation*}
 \label{E:}
\Omega^s|_{V_{\alpha\beta}} = 
 \left\lbrace
\begin{array}{cccrl}
dx^\alpha &       & -\;\;\;\; dx^n & \text{ \;si }  & s =  \alpha, \vspace{0.08cm}\\
& dx^\beta&-\;\;\;\; dx^n & \text{\; si }  & s = \beta, \vspace{0.08cm}\\
& &\:\:\:\;\;\;\; dx^n & \text{\; si }  & s =  n,\vspace{0.08cm}\\  
\Omega_{ \alpha}^s dx^\alpha\;\;\; +
    & \Omega_{ \beta}^sdx^\beta &-\;\;\;\; dx^n & \text{\; si }  & s \in  \{n+1,n+2\}.
\end{array}
\right.
\end{equation*}

Puisque  c'est une intersection d'hypersurfaces totalement géodésiques, la sous-variété $V_{{\alpha\beta}}$ est totalement géodésiques pour $\Pi$ (d'après le \LL \ref{L:projint}). On peut donc  considérer la restriction $\Pi_{\alpha\beta}=\Pi|_{V_{\alpha\beta}}$ de $\Pi$ à $V_{\alpha\beta}$. Clairement,  $\Pi_{\alpha\beta}$ est admissible pour $W_{\!{\alpha\beta}}$ et par voie de  conséquence,  quelles que soient les constantes $\tau^1,\ldots,\tau^{n-1}$, les applications 
$$ 
(t^\alpha,t^\beta)\longmapsto Z_p\big(\tau^1,\ldots,\tau^{\alpha-1},t^\alpha,\tau^{\alpha+1},\ldots,     
\tau^{\beta-1},
t^\beta,\tau^{\beta+1},\ldots,\tau^{n-1}\big)
$$
(pour $p\in \{ \alpha,\beta,n,n+1,n+2\}$) vérifient le système formé par  les restrictions à $V_{\alpha\beta}$ des équations (\ref{E:eee}). 
On vérifie  immédiatement que celui-ci n'est rien d'autre que  le système (\ref{E:eee}) associé à la connexion projective $ \Pi_{\alpha\beta}$.  
 On se trouve donc ramené au cas $n=3$ que l'on a pu traiter de manière directe. 
 On montre ainsi que 
$ A^\gamma=B^\gamma_\beta=C^\gamma_{\delta\epsilon}= D_{\delta\epsilon}=0  $
pour tout  $\gamma,\delta,\epsilon \in \{ \alpha,\beta\}$. Cela étant valable quels
que soient $\alpha$ et $\beta$, il en découle que (\ref{E:eee}) n'admet que la solution triviale. On en déduit la proposition. 
\end{proof}
Soit $W$ un $(n+2)$-tissu. Notons $\Pi_W$ la connexion projective admissible par $W$. Si $\varphi$ désigne un (germe de) biholomorphisme, il est immédiat que $\varphi^*(\Pi_W)$ est admissible par $\varphi^*(W)$. Par unicité, il en découle  que 
$$ \varphi^*(\Pi_W)= \Pi_{ \varphi^*(W) } \, . $$
En d'autres termes: $\Pi_W$ est canoniquement attachée à $W$. 
\begin{defi}
 La connexion projective $\Pi_W$ est la {\it connexion projective canonique} du tissu $W$. La {\it courbure de Cartan} de $W$ notée $\mathfrak C_W$ est la courbure $d\omega+ \omega\wedge  \omega \in \Omega^2_M \otimes \mathfrak{sl}_{n+1}(\mathbb C)$ de la connexion projective normale associée à $\Pi$. 
\end{defi}
De la  \PP \ref{P:existePI}, il vient que $\mathfrak C_W$ est attachée à $W$ de façon invariante:  on a $$ \varphi^*(\mathfrak C_W)= \mathfrak C_{ \varphi^*(W) }\,  $$ 
pour tout (germe de) biholomorphisme $\varphi$.

Vu l'\EE \ref{Ex:platelin}, la connexion projective canonique d'un $(n+2)$-tissu linéaire est la connexion projective plate.  Comme conséquence, il vient que $\Pi_W$ est intégrable si $W$ est linéarisable. La réciproque est tout aussi immédiate. On a donc le 
\begin{theo} 
\label{T:MAIN}
Les conditions suivantes sont équivalentes: 
\begin{enumerate}
 \item le $(n+2)$-tissu $W$ est linéarisable;
\item la connexion projective $\Pi_W$ est intégrable;
\item la courbure de Cartan $\mathfrak C(W)$ est identiquement nulle.
\item le tenseur de Weyl $\mathfrak W_{\Pi_W}$ est identiquement nul (lorsque $n>2$);
\item le tenseur de courbure  $\mathfrak P _{\Pi_W}$ est identiquement nul (lorsque $n=2$).
\end{enumerate}
\end{theo}

\subsection{Linéarisabilité  des $d$-tissus en dimension $n$}
Il est  facile d'adapter les résultats de la section précédente au cas des tissus de codimension 1 formés de plus de  $n+2$ feuilletages. On fixe $d\geq n+2$ dans ce qui suit ainsi qu'un $d$-tissu en hypersurfaces  $W_d=(\mathcal F^1,\ldots,\mathcal F^d)$ sur $U$. Pour   $\ell$ compris entre  $n+2$ et $d$,  on pose  
$$W(\ell)=(\mathcal F^1,\ldots,\mathcal F^{n+1},\mathcal F^\ell).$$
Chacun des $W(\ell)$ est un $(n+2)$-tissu sur $U$. On note $\Pi(\ell)$ sa 
connexion projective canonique. Il est clair que  $W$ est linéarisable si et seulement si il existe une connexion projective intégrable admissible pour $W$ et donc pour tous les $W(\ell)$. On en déduit  le 
\begin{coro} 
\label{C:main2}
Le $d$-tissu $W_d$ est linéarisable si et seulement si 
\begin{enumerate}
\item la connexion projective $\Pi(n+2)$ est plate;
 \item les connexions projectives $\Pi(n+2),\ldots, \Pi(d)$ coïncident.
\end{enumerate}
\end{coro}
Il n'est pas difficile d'exprimer la condition  (2) de ce corollaire en terme d'invariants différentiels scalaires. 
Pour tout $i,j,k=1,\ldots,n$ et quel que soit $\ell$ 
compris entre  $n+2$ et $d$, notons $ \Pi(\ell)^k_{ij}$ (resp. 
$\overline{\Pi}(\ell)^k_{ij}$) 
les coefficients de Thomas de la connexion projective $\Pi(\ell)$ relativement à un système  de coordonnées $x^1,\ldots,x^n$ (resp. $\overline{x}^1,\ldots, \overline{x}^n$) sur $U$.
 Quel que soit $\ell$ strictement plus grand que $n+2$ et pour tout $i,j,k=1,\ldots,n$, posons 
$$ \Sigma(\ell)^k_{ij}= \Pi(\ell)^k_{ij}-\Pi(n+2)^k_{ij} \qquad \qquad  \mbox{ et } \qquad \qquad 
\overline{\Sigma}(\ell)^k_{ij}= \overline{\Pi}(\ell)^k_{ij}-\overline{\Pi}(n+2)^k_{ij}\,. 
$$

 En utilisant (\ref{E:transocoeffTHOMAS}), on établit que les quantités 
$ \Sigma(\ell)^k_{ij}$ et $\overline{\Sigma}(\ell)^k_{ij}$ s'expriment les unes en fonction des autres via  les formules  de transformation 
$$  \overline{\Sigma}(\ell)^{\gamma}_{\alpha\beta}=
\Sigma(\ell)^k_{ij}
\frac{\partial x^i}{ \partial \overline{x}^\alpha}
\frac{\partial x^j}{ \partial \overline{x}^\beta}
\frac{ \partial \overline{x}^\gamma}{\partial x^k}\, .
 $$

En d'autres termes, les quantités $ \Sigma(\ell)^k_{ij}$ sont les composantes (dans les coordonnées $x^1,\ldots,x^n$) d'un tenseur 
\begin{equation*}
 \Sigma(\ell)\in T_U\otimes {\rm Sym}^2( \Omega^1_U)
\end{equation*}
 canoniquement associé à la paire $(W(n+2),\mathcal F_\ell)$. Par définition, $\Sigma(\ell)$ est identiquement nul si et seulement si les connexions projectives $\Pi(n+2)$ et $\Pi(\ell)$ coïncident. On peut alors reformuler le \CC \ref{C:main2} en des termes purement invariants:
\begin{coro} 
Le $d$-tissu $W_d$ est linéarisable si et seulement si
\begin{enumerate}
\item la courbure de Cartan $\mathfrak C\big(W(n+2)\big)$ est identiquement nulle;
\item les tenseurs  $\Sigma(n+2),\ldots,\Sigma(d)$ sont tous identiquement nuls.
\end{enumerate}
\end{coro}

On  en déduit une généralisation en dimension $n$ arbitraire d'un résultat classique de \cite[\S29]{bb}:
\begin{coro}
\label{C:nombrelin}
 Soit $d\geq n+2$. Modulo les transformations projectives, 
 un $d$-tissu en hypersurfaces dans un espace de dimension $n$ admet au plus une linéarisation. 
\end{coro}
\begin{proof} 
 Il faut (et il suffit de) montrer que si  $W_d$ et $\varphi^*(W_d)$ sont des tissus linéaires,  alors $\varphi$ est la restriction d'une transformation projective.  Par hypothèse, la connexion projective plate est admissible par $W_d$ et  par $\varphi^*(W_d)$. Chacun d'eux  admettant au plus une connexion admissible, on en déduit que $\varphi$ laisse invariante la stucture projective canonique de $\mathbb P^n$. En  termes plus élémentaires: $\varphi$ transforme les droites projectives en droites projectives. Il est classique  qu'un tel biholomorphisme  n'est rien d'autre que la restriction d'une transformation projective, d'où le corollaire. 
\end{proof}

\begin{rema}
 L'énoncé du corollaire précédent n'est pas valide lorsque $d=n+1$. En effet, soit $C$ une courbe algébrique dans $\mathbb P^n$, irréductible, non-dégénérée et de degré $n+1$. On sait qu'alors son genre arithmétique vérifie $p_a(C)\leq 1$. Supposons qu'il y ait égalité (de telles courbes existent quel que soit $d$; $C$ est alors dite {\og \!de Castelnuovo\!\fg}). Via la dualité projective, il est classique d'associer à $C$ un $(n+1)$-tissu linéaire  $W_C$ {\og \!algébrique\!\fg} de codimension 1 défini sur un ouvert de Zariski de l'espace projectif dual $ \check{\mathbb P}^n$.  Par le théorème d'addition d'Abel,  la trace (par rapport à la série linéaire hyperplane) de ``la'' 1-forme régulière non-nulle sur $C$  est identiquement nulle. On en déduit que le rang de $W_C$ est 1, c'est-à-dire que l'on peut trouver  des intégrales premières (locales) $u_1,\ldots,u_{n+1}$ des feuilletages  qui composent $W_C$ (localement) telles que  $ u_1+\cdots +u_{n+1}\equiv 0$. On vérifie alors que l'application $U=(u_1,\ldots,u_{n})$ tranforme $W_C$ en le $(n+1)$-tissu $W_0$ formé des pinceaux d'hyperplans parallèles  $\{x^i=cst.\}$ (pour $i=1,\ldots,n$) et 
$\{x^1+\cdots+x^n=cst.\}$.  En utilisant le théorème d'Abel et sa réciproque, on montre ainsi qu'il y a une correspondance biunivoque entre l'espace des linéarisations locales  de $W_0$ modulo composition à gauche par une  transformation projective et l'espace des courbes algébriques réduites de $\mathbb P^n$, non dégénérées, de degré $n+1$ et de genre arithmétique 1. Comme conséquence directe vient que $W_0$ admet un nombre infini (une famille continue en fait) de linéarisations projectivement indépendantes. 

Dans le cas des tissus plans ($n=2$), Gronwall a  conjecturé en 1912 qu'un 3-tissu admet au plus une classe d'équivalence projective de linéarisations lorsqu'il est de rang zéro. Si l'on sait maintenant que ce nombre de classes  de linéarisations est fini (voir \cite{boruvka,vaona,smirnov,grifone,goly}), la {\og Conjecture de Gronwall \fg} n'est toujours pas démontrée à ce jour. On peut la généraliser à la dimension $n>1$ arbitraire: un $(n+1)$-tissu en hypersurfaces non-parallélisable admet-il  au plus une classe d'équivalence projective de linéarisations? Il semble que l'on puisse se ramener au cas plan par tranchage
\end{rema}

\section{Généralisations}
\label{S:gen}

\subsection{Un critère général de linéarisation}
\label{S:critergeneraldeLIN}
La méthode qui nous a permis d'obtenir le critère de linéarisabilité du \TT \ref{T:MAIN} pour les tissus en hypersurfaces est en fait très générale puisqu'elle s'applique à des objets plus généraux que les tissus de codimension 1. 

Soient $a_1,\ldots,a_{n-1}$  des entiers strictements positifs. Posons $\underline{a}=(a_1,\ldots,a_{n-1})$ et $d=\sum_{i=1}^{n-1} a_i$. 
Décidons d'appeler un {\it $\underline{a}$-tissu généralisé} sur $U\subset \mathbb C^n$ la donnée notée $W_{\underline{a}}$ de $d$-feuilletages ${\mathcal F}_1,\ldots,\mathcal F_d$ sur $U$ dont exactement $a_c$ sont de codimension $c$, cela pour tout entier $c$ compris entre 1 et $n-1$. C'est une notion plus générale que celle de tissu mixte: on ne demande pas forcément que  les espaces tangents des feuilletages de $W_{\underline{a}}$ soient en position générale. 

Cette définition étant posée, il se trouve que l'approche utilisée pour établir le \TT \ref{T:MAIN} s'adapte quasiment sans changements à la question de caractériser les tissus généralisés linéarisables, puisqu'il est toujours aussi évident que 
\begin{quote}
 {\it un tissu généralisé est linéarisable si et seulement si il existe une connexion projective intégrable qui lui est compatible.}
\end{quote}

La question de l'intégrabilité d'une connexion projective compatible avec un tel $W_{\underline{a}}$ étant classique et bien comprise, le problème de caractériser les tissus généralisés linéarisables se ramène donc à la question de savoir s'il existe des coefficients de Thomas $\Pi_{ij}^k$ tels que  pour tout $m\in \{1,\ldots,n-1\}$, les feuilles de dimension $m$ d'un tel tissu  $W_{\underline{a}}$ donné sont solutions du système différentiel $(S_\Pi^m)$ de la \PP \ref{P:11projo}. C'est une question d'algèbre linéaire sur l'anneau de fonctions holomorphes $\mathcal O_U$. 
 En effet, ayant fixé un système de coordonnées $x^1,\ldots,x^n$ sur $U$, une connexion projective $\Pi$ sur $U$ est définie par ses coefficients de Thomas. L'espace $\{ (\Pi_{ij}^k)_{i,j,k=1}^n \, \}$ de ces coefficients  est un $\mathcal O_U$-module libre de rang $(n-1)n(n+2)/2$. Un simple décompte montre que la condition pour que les feuilles d'un feuilletage de codimension $c$ sur $U$ soient totalement géodésiques pour $\Pi$ consiste en 
\begin{equation*}
 \label{E:pnk}
 p(n,c)=\frac{c (n-c)(n-c+1)}{2}
\end{equation*}
équations scalaires non-homogènes en les  $\Pi_{ij}^k$ (à coefficients dans $\mathcal O_U$). La condition pour qu'une connexion projective $\Pi$ soit admissible par un tissu généralisé $W_{\underline{a}}$ va donc consister en un système $S(W_{\underline{a}})$  formé de 
$$\mu(\underline{a})=\sum_{c=1}^{n-1} p(n,c)\, a_c $$
équations scalaires du même type. Si les entiers $a_1,\ldots,a_{n-1}$  sont tels que 
$\mu(\underline{a})   \geq \frac{1}{2}{(n-1)n(n+2)}$, l'on s'attend à ce que le système $S(W_{\underline{a}})$ soit surdéterminé et donc n'admette au plus qu'une seule solution. Si une solution existe et est unique, on la note $\Pi_{W_{\underline{a}}}$. Elle est alors canoniquement attachée à $W_{\underline{a}}$ (par unicité) et comme pour les tissus de codimension 1,  $W_{\underline{a}}$ est linéarisable si et seulement si la courbure $\mathfrak C( \Pi_{W_{\underline{a}}})$ est identiquement nulle.

\subsection{Exemples de tissus  en courbes non-linéarisables}
Afin d'illustrer l'approche présentée ci-dessus, 
 donnons nous un  paramètre  $\epsilon \in \mathbb C$ 
et considérons le 8-tissu en courbes $W_8^\epsilon$ dans $\mathbb C^3$   dont les feuilles sont les courbes intégrales des champs de vecteurs 
$\frac{\partial }{\partial x^1}+Z_i^2\frac{\partial }{\partial x^2}+Z_i^3\frac{\partial }{\partial x^3}$
 dont les composantes $Z_i=(1,Z_i^2,Z_i^3)$ sont 
\begin{align*}
Z_1=& \, (1,0,0)
&&  Z_2= (1,1,0)
&& Z_3= (1,0,1) &&  Z_4= (1,1,1) \\
Z_5=& (1,-1,1)&& 
Z_6= (1,1,-1)  && 
Z_7= (1,-1,-1)&& 
Z_8=  \Big(1,\epsilon\frac{y}{x} ,   \frac{z}{x}\Big)\,.
\end{align*}

Une condition nécessaire pour que  $W_8^\epsilon$ soit linéarisable est qu'il existe une connexion projective qui lui soit compatible. De façon plus explicite, cela équivaut à l'existence de fonctions  holomorphes $A^k,B^k_i, C^k_{ij},D_{ij}$ pour $i,j,k=2,3$, symétriques en les indices inférieurs $i$ et $j$, telles que les feuilles de $W_8^\epsilon$ soient toutes solutions des  équations différentielles du second  ordre (pour $\ell=2,3$): 
\begin{equation*}
\qquad
\frac{d^2 Z^\ell}{dt^2}= A^\ell+\sum_{i=2}^3 B^\ell_{i}
\frac{d Z^i}{dt}
+ \sum_{i,j=2}^3 C^\ell_{ij} \frac{d Z^i}{dt}\frac{d Z^j}{dt}
+\frac{d Z^\ell}{dt} \bigg(
\sum_{i,j=2}^3 D_{ij}\frac{d Z^i}{dt}\frac{d Z^j}{dt} \bigg)\,. 
\end{equation*}

Un calcul immédiat montre que de telles fonctions $A^k,B^k_i, C^k_{ij},D_{ij}$ existent si et seulement le paramètre  $\epsilon$ vaut $1$. En conséquence, le tissu $W_8^\epsilon$ n'est pas linéarisable si  $\epsilon\neq 1$. C'est en fait une condition nécessaire et suffisante dans cet exemple: on vérifie immédiatement que $W_8^1$ est linéarisable (il est linéaire!).

\subsection{Un exemple sur les espaces de modules de courbes de genre 0} Soit $m$ un entier plus grand que 4.  L'on désigne par ${M}_{0,m}$ l'espace de module des courbes algébriques de genre $0$ avec $m$ points distincts marqués et numérotés: en d'autres termes, c'est l'espace des configurations projectives de $m$ points distincts de ${\mathbb P}^1$. C'est une variété quasi-projective lisse de dimension $m-3$.  

L'opération  d'oubli d'un point d'une configuration induit  un morphisme 
${M}_{0,m}\rightarrow {M}_{0,m-1}$. Un élément de 
${M}_{0,m}$ étant composé de $m$ points, il existe 
 $m$ tels morphismes d'oublis sur ${M}_{0,m}$. Plus généralement, pour $k$ compris entre 1 et $m-1$, il existe ${ m \choose k } $ morphismes d'oublis de $k$ points ${M}_{0,m}\rightarrow {M}_{0,m-k}$. Si $k<m-3$, les ensembles de niveaux d'un tel morphisme d'oublis  forment un feuilletage de codimension $m-3-k$ sur ${M}_{0,m}$. On vérifie sans difficulté que les feuilletages de codimension $k$ induits par deux de ces morphismes  sont différents. Par conséquent, pour tout $k\in \{1,\ldots,m-4\}$, l'ensemble des feuilletages associés aux oublis de $k$ points définissent un ${ m \choose k }$-tissu généralisé de codimension $c=m-3-k$ sur 
${M}_{0,m}$, que l'on note $W^{c}(M_{0,m})$. \medskip

Le plus élémentaire de ces tissus est  le 5-tissu plan $W^1(M_{0,5})$ qui n'est rien d'autre que le célèbre {\og tissu de Bol\fg} considéré  dans \cite{bol,gelfandmacpherson,damiano}. Les tissus généralisés en hypersurfaces $W^1(M_{0,m})$    naturellement définis sur les espaces de modules $M_{0,m}$ ont été considérés en premier lieu par Burau dans  \cite{burau} (sous une forme différente mais équivalente). Les $m$-tissus en courbes $W^{m-4}(M_{0,m})$  sont introduits et étudiés dans \cite{damiano}, où il est montré que ceux-ci sont de rang maximal ({\it cf.} \cite{damiano} pour des précisions). Ces résultats font appara\^itre les tissus  $W^c(M_{0,m})$ (pour $m\geq 3$ et $c=1,\ldots,m-4$) comme des tissus potentiellement intéressants du point de vue de leur rang et de leurs relations abéliennes, que l'on peut suspecter être liées aux polylogarithmes. 
\smallskip

On  s'intéresse ci-dessous à la question de savoir si certains des $W^c(M_{0,m})$ sont linéarisables. Plus précisément, nous montrons que les tissus en hypersurfaces $W^1(M_{0,m})$ ne le sont pas.  On notera que le \TT \ref{T:MAIN} ne s'applique pas aux tissus $W^1(M_{0,m})$. En effet, ce sont des tissus généralisés: leurs normales  ne vérifient pas  l'hypothèse de position générale définie dans la Section \ref{S:notations} (sauf lorsque $m=5$). \smallskip

L'étude de la linéarisabilité des tissus $W^1(M_{0,m})$ est rendue particulièrement facile si l'on applique la proposition suivante qui est une conséquence immédiate des résultats de la Section \ref{S:restriction}.
\begin{prop}
\label{P:nn}
Si un $d$-tissu $W_d$  en hypersurfaces est linéarisable, alors tout tissu obtenu par tranchage de $W_d$ le long des feuilles de l'un ou plusieurs des   feuilletages qui le constituent est aussi linéarisable. 
\end{prop}

Il est classique que le tissu de Bol $W^1(M_{0,5})$ n'est pas linéarisable ({\it cf.} \cite[p. 262]{bb}). On suppose donc $m>5$ dans  ce qui suit. 
Les applications d'oublis de $m-4$ points $M_{0,m}\rightarrow M_{0,4}$ sont autant d'intégrales premières des 
feuilletages de $W^1(M_{0,m})$. Chacune se factorise (de plusieurs façons possibles) par une application d'oublis de $m-5$ points:
\[
\xymatrix@R=0.1cm@C=0.8cm{  
 M_{0,m}  \ar[rr] \ar[dr]  &  &  M_{0,4}\, .\\
  &  M_{0,5}  \ar[ru]&  }\]
On en déduit que $W^1(M_{0,m})$ admet comme sous-5-tissu le tissu   $\pi^* W^1(M_{0,5})$ obtenu en tirant en arrière le tissu de Bol $W^1(M_{0,5})$  par une application d'oublis $\pi:M_{0,m}\rightarrow M_{0,5}$. Si l'on considère alors le tissu obtenu en prenant la restriction de  $W^1(M_{0,m})$ le long d'un  intersection convenable de $m-5$ de ses feuilles, on obtient un tissu plan qui admet $W^1(M_{0,5})$ comme sous-5-tissu. Celui-ci n'étant pas linéarisable, on peut conclure en appliquant la \PP \ref{P:nn}.
\begin{prop}
Aucun des tissus  $W^1(M_{0,m})$ n'est linéarisable. 
\end{prop}


Combiné avec les résultats récents \cite{cl} de Cavalier et Lehmann, 
ce résultat ouvre une perspective intéressante dans l'étude des {\og tissus exceptionnels\fg}. Par des calculs directs, nous avons vérifié (avec J.V. Pereira) que le tissu $W^1(M_{0,6})$ est bien de rang maximal: la dimension de l'espace de  ses relations abéliennes prend la valeur maximale possible  pour un 15-tissu  régulier (au sens de \cite{cl}) dans un espace de dimension 3, à savoir la borne $\pi'(15,3)=26 $  de Cavalier et Lehmann.  Ce tissu n'étant pas linéarisable, il n'est pas équivalent à un tissu algébrique, c'est-à-dire  un tissu associé via la dualité projective à une courbe algébrique gauche.  Suivant une terminologie classique en géométrie des tissus, on dira que $W^1(M_{0,6})$ est {\it exceptionnel}. Il serait intéressant de savoir si tous les $W^1(M_{0,m})$ le~sont~aussi et, si c'est le cas, de décrire explicitement leurs relations abéliennes.\footnote{J.V. Pereira nous informe avoir démontré que les $W^1(M_{0,m})$ sont en effet tous de rang maximal, voir \cite{jorge}.}


\subsection{Un critère de linéarisation pour certains  tissus mixtes dans $\mathbb C^3$} 
Suivant la méthode décrite dans la section \ref{S:critergeneraldeLIN} ci-dessus, nous avons obtenu un critère de linéarisation pour certains tissus mixtes en dimension 3: 
\begin{prop} Dans un espace de dimension 3, il existe une unique connexion projective $\Pi_W$  compatible avec un 6-tissu mixte $ W$  formé de trois feuilletages en hypersurfaces et de trois feuilletages en courbes. Par conséquent,  $W$ est linéarisable si et seulement si la courbure  $\mathfrak C(\Pi_W)$ de $\Pi_W$ est identiquement nulle. 
\end{prop}

La preuve repose sur un calcul formel direct. 
Soient $\omega_1,\omega_2,\omega_3$ des  1-formes différentielles et 
$X_1,X_2,X_3$ des champs de vecteurs en trois variables $x,y,z$, 
 définissant des feuilletages non-singuliers 
en l'origine de $\mathbb C^3$. On note $W$ le 6-tissu mixte généralisé  défini par ces feuilletages (dont trois sont en hypersurfaces et trois sont des feuilletages par courbes). Nous avons vérifié  (par un calcul en MAPLE) que le déterminant du système linéaire caractérisant les connexions projectives sur $(\mathbb C^3,0)$ compatibles avec $W$ est égal  à 
\begin{equation*}
\Delta=-512\,\Big( 
\frac{\omega_1\wedge \omega_2\wedge \omega_3
}{dx\wedge dy\wedge dz}
\Big)^3
\Big( 
\frac{X_1\wedge X_2\wedge X_3
}{\partial_x\wedge \partial_y\wedge \partial_z}
\Big)^2
\prod_{ i,j=1} ^3 \omega_i(X_j) \,.
\end{equation*}

Puisque $\Delta$ ne s'annule pas si et seulement si les feuilletages qui composent $W$ sont en position générale (au sens de la définition de la Section \ref{S:notations}), cela démontre la proposition ci-dessus.

\subsection{Remarques au sujet du cas général} 
Notons encore $W_{\underline{a}}$ un tissu formé de $a_c$ feuilletages de codimension $c$ pour $c=1,\ldots,n-1$.  
Comme on l'a dit plus haut, l'existence et l'unicité d'une connexion projective (locale) compatible avec $W_{\underline{a}}$ se traduit par un problème d'algèbre linéaire avec autant d'équations que d'inconnues lorsque $\mu(\underline{a})=(n-1)n(n+2)/2$. Sous cette hypothèse, il est tentant de conjecturer que si $W_{\underline{a}}$ est un tissu mixte dont les feuilletages sont en position générale, ce problème admet une unique solution.

Des calculs effectifs nous ont montrés que cela n'est pas vrai dans tous les cas. Par exemple, en dimension 4, nous avons vérifié qu'il existe une famille continue de connexions projectives  compatibles avec un 6-tissu dans $\mathbb C^4$ formé de cinq feuilletages en hypersurfaces et d'un feuilletage de codimension 2 (cas $n=4$ et $\underline{a}=(5,1,0)$). \medskip

Une autre situation se présente: celle où il y a existence et unicité d'une  connexion projective compatible avec un tissu donné $W_{\underline{a}}$ générique, sans que la notion de {\og généricité\fg} impliquée ne soit reliée (semble-t-il) au fait que les feuilletages de $W_{\underline{a}}$ soient en position générale.  Par exemple, nous avons vérifié  qu'un 12-tissu en courbes dans $\mathbb C^4$ est compatible avec une unique connexion projective s'il est {\og suffisament générique\fg}.

\section{Quelques remarques historiques}
\label{S:histoire}
En guise de conclusion, nous voudrions revenir sur l'histoire  du problème de la linéarisation des tissus. \smallskip 

Comme cela est évoqué dans l'introduction, la question de la linéarisabilité d'un tissu s'est d'abord posée relativement à des questions en {\og Théorie des Abaques\fg} ou {\og Nomographie\fg}. C'est  Lalanne dans \cite{lalanne} qui, le premier, a pensé à simplifier la lecture des abaques en trois variables $A(x,y,z)=0$ en les rectifiant (quand c'est possible) au moyen de {\og tranformations anamorphiques\fg}, c'est-à-dire de changements de coordonnées ponctuelles de la forme $(x,y)\mapsto  
(\overline{x}(x),\overline{y}(y))$. C'est ensuite  Massau dans \cite{massau} qui a compris l'intérêt de se permettre les transformations ponctuelles $(x,y)\mapsto  (\overline{x}(x,y),\overline{y}(x,y))$ les plus générales pour essayer de rectifier une abaque en trois variables ou, en termes plus modernes, pour chercher à  linéariser un 3-tissu~plan. 
\smallskip

C'est apparemment  en relation avec des questions d'algébrisation des tissus que Blaschke et ses collaborateurs ont étudié la question de la linéarisabilité des tissus plans. La Section \S29 du livre \cite{bb} y est entièrement consacrée.  Les auteurs y montrent que, un système de coordonnées $x,y$ étant choisi, les feuilles d'un 4-tissu plan $W_4$ donné sont aussi des courbes intégrales 
d'une unique équation différentielle du second ordre 
de la forme 
\begin{equation}
\label{E:equatsecondordre}
  y''=A(y')^3+B(y')^2+Cy'+D\,,\end{equation}
forme   invariante par les changements de coordonées ponctuelles.  L'équation (\ref{E:equatsecondordre}) est donc  canoniquement attachée à $W_4$.  Les auteurs en déduisent plusieurs résultats sur les 4-tissus plans. 
En effet,  on peut lire
\begin{quote}
 {\og {\it Dann und nur dann lä{\ss}t sich ein 4-Gewebe geradlinig machen, wenn das zugehörige quasigeodätische System mit den Geraden der Ebene topologisch äquivalent ist}\fg}
\end{quote}
en bas de la page 247 (ce qui est un critère de linéarisation d'un 4-tissu plan) et juste en dessous 
\begin{quote}
  {\og {\it Eine toplogische Abbildung, die vier Geradenscharen  in vier Geradenscharen überführt, ist \\ ${}$ \; sicher projektiv}\fg}
\end{quote}
(ce qui est notre \CC \ref{C:nombrelin} dans le cas $n=2$). 

Le critère caractérisant l'équivalence d'un {\og {\it quasigeodätissche System}\fg} avec celui des droites du plan   n'appara\^it pas explicitement dans le corps du texte de la section \S29, mais  est clairement évoqué dans le cinquième des {\og Aufgaben und Lehrs\"atze zu \S29\fg} page 249. Le critère de Tresse y est juste mentionné, les auteurs renvoyant  à l'{\og Aufgabe 2\fg} de la section \S22  où les références \cite{tresse,bollin,dubourdieu,cartan} sont données.  \`A noter aussi que les géomètres allemands connaissaient l'interprétation géométrico-invariante du critère caractérisant la platitude de l'équation (\ref{E:equatsecondordre})  en termes de connexion projective ({\og projektiven Zusammenhang\fg}) comme le montre la lecture de l'{\og Aufgabe 11\fg}, page 249. 

Vu la description  des résultats de \cite{bb} faite ci-dessus, il appara\^it comme certain que les géomètres allemands qui étudiaient les tissus  vers  1930 connaissaient un critère  géométrique s'exprimant en termes d'invariants différentiels caractérisant les tissus plans linéarisables\footnote{C'est d'ailleurs en faisant référence aux résultats de la section \S29 qu'est démontré dans le second paragraphe de la section \S31 de \cite{bb} que le  5-tissu de Bol n'est pas linéarisable.}.  
  Ces résultats ont été oubliés par~la~suite, semble-t-il. \medskip

Récemment,  Hénaut a abordé la question de la linéarisation des tissus plans dans \cite{henaut}. S'appuyant sur les travaux \cite{lie, tresse} de Lie et Tresse, il obtient un critère caractérisant les tissus plans linéarisables qui est essentiellement le même que celui obtenu par les géomètres allemands. Hénaut a  travaillé indépendamment, sans connaître leurs résultats: s'il cite le livre \cite{bb}, c'est pour dire que seuls des cas particuliers y ont été traités ({\it cf.} \cite{henaut} page 531). 
\medskip

Dernièrement en 2004, Akivis, Goldberg et Lychagin reviennent sur la linéarisabilité des tissus plans dans \cite{akgoly}. Insistant sur le fait que les résultats de Hénaut énoncés dans \cite{henaut} ne sont pas formulés en des termes invariants, il reprennent une approche proposée en 1973  par Akivis  et obtiennent une caractérisation des tissus plans linéarisables qui s'exprime par l'annulation de certains invariants différentiels explicites.  Ils disent  résoudre une conjecture  posée par   Blaschke 
à la fin de la section \S42 de \cite{blaschke}, où l'on peut lire
\begin{quote}
 {\og {\it Es scheint also möglich zu sein, eine beliebige Wabe in der Umgebung dritter Ordnung eines Punktes durch eine geradlinige Wabe zu ersetzen, während eine Annäherung in vierter Ordnung im allgemeinen nicht mehr möglich ist. Für die {\og Streckbarkeit\fg} einer $W_4$ haben wir also zwei Bedingungen vierter Ordnung zu erwarten.} \fg}
\end{quote}

Bien que cela ne soit pas complètement explicite,  Blaschke semble bien conjecturer à cet endroit quelle doit être la forme que l'on peut attendre d'un critère caractérisant les 4-tissus plans linéarisables. Cela est surprenant, puisque, comme on l'a dit plus haut, un tel critère était déja présenté dans \cite{bb},  livre  datant de 1938 dont il est l'un des deux auteurs. \medskip

Enfin, alors que nous terminions la préparation de cet article, Goldberg et Lychagin ont rendu disponible la prépublication \cite{goly2} où ils   abordent le problème de la linéarisation des tissus de codimension 1 en dimension arbitraire.





\bibliographystyle{cdraifplain}

\end{document}